\documentclass[a4paper,12pt]{article}
\usepackage[in]{fullpage}

\usepackage{amsmath, amsthm, amssymb}
\usepackage{mathrsfs}
\usepackage{mathtools}

\usepackage{cases}
\usepackage{graphicx}

\title{Crossed homomorphisms and low dimensional representations of mapping class groups of surfaces}
\author{Yasushi Kasahara}
\date{\today}

\numberwithin{equation}{section}

\begin{document}

\newtheorem{theorem}{Theorem}[section]
\newtheorem{main_thm}[theorem]{Theorem}
\newtheorem{lemma}[theorem]{Lemma}
\newtheorem{corollary}[theorem]{Corollary}
\newtheorem{prop}[theorem]{Proposition}

\theoremstyle{definition}
\newtheorem{definition}[theorem]{Definition}
\newtheorem*{fact}{Fact}
\newtheorem*{claim}{Claim}

\newtheorem{example}[theorem]{Example}
\newtheorem{problem}[theorem]{Problem}
\theoremstyle{remark}
\newtheorem{remark}[theorem]{Remark}

\newcommand{\ppar}{\par\goodbreak\medskip} 
\newcommand{\Z}{{\mathbb Z}}
\newcommand{\im}{\operatorname{Im}}
\newcommand{\Q}{\mathbb Q}
\newcommand{\C}{\mathbb C}
\newcommand{\M}[1]{\operatorname{Mod}( #1)}
\newcommand{\mg}{\mathcal{M}_g}
\newcommand{\I}{\mathcal{I}}
\newcommand{\Dm}[1]{\dim{\left( #1 \right)}}
\newcommand{\Sp}[1]{\operatorname{Sp}\left( #1 \right)}
\newcommand{\GL}[1]{\operatorname{GL}{\left( #1 \right)}}
\newcommand{\Hom}{\operatorname{Hom}}
\newcommand{\rank}{\operatorname{rank}}
\newcommand{\symp}{\rho_0} \newcommand{\trivial}{\varepsilon}
\newcommand{\FH}{H_1(S;\C)}
\newcommand{\cc}[1]{Z^1\left(\M{S}; #1 \right)}
 \newcommand{\ccC}{\cc{H_\C}}
 \newcommand{\ccZ}{\cc{H}}
\newcommand{\cb}[1]{B^1\left(\M{S}; #1 \right)}
 \newcommand{\cbC}{\cb{H_\C}}
 \newcommand{\cbZ}{\cb{H}}
\newcommand{\co}[1]{H^1\left( \M{S}; #1 \right)}
 \newcommand{\coC}{\co{H_\C}}
 \newcommand{\coZ}{\co{H}}
\newcommand{\eigensp}[2]{E_{#1}^{#2}}
  \newcommand{\ea}[1]{\eigensp{#1}{a}}
  \newcommand{\eb}[1]{\eigensp{#1}{b}}
  \newcommand{\ela}{\ea{\lambda}}
  \newcommand{\elb}{\eb{\lambda}}
\newcommand{\Uh}{\widehat{U}}
\newcommand{\Ab}{\bar{A}}
\newcommand{\Bb}{\bar{B}}
\newcommand{\Rb}{\bar{R}}
\newcommand{\Sb}{\bar{S}}
\newcommand{\Xb}{\bar{X}}
\newcommand{\At}{\widetilde{A}}
\newcommand{\Bt}{\widetilde{B}}
\newcommand{\Ft}{\widetilde{F}}
\newcommand{\Xt}{\widetilde{X}}
\newcommand{\Gt}{\widetilde{G}}
\newcommand{\Pt}{\widetilde{P}}
\newcommand{\Qt}{\widetilde{Q}}
\newcommand{\diag}[1]{\operatorname{diag}( #1 )}
\newcommand*{\transpose}[2][-3mu]{\ensuremath{\mskip1mu\prescript{\smash{\mathrm t\mkern#1}}{}{\mathstrut#2}}}\newcommand{\vect}[1]{\boldsymbol{#1}}
\newcommand{\Fix}[1]{\operatorname{Fix}\left( #1 \right)}
\newcommand{\rotation}{r} \newcommand{\XXX}{\begin{matrix}
		  1 & 1 & 0 & \alpha \\
		  0 & 1 & 0 &  0 \\
		  0 & \frac{1}{\alpha} & 1 & 1 \\
		  0 & 0 & 0 & 1		  
		 \end{matrix}}

\bibliographystyle{amsplain}

\maketitle
\begin{abstract}
 We continue the study of low dimensional linear representations of mapping class groups of surfaces
initiated by Franks--Handel and Korkmaz. We consider $(2g+1)$-dimensional 
complex linear representations of the pure mapping class groups of compact orientable surfaces of genus $g$. 
We give a complete classification of such representations for $g \geq 7$ up to conjugation, 
in terms of certain 
twisted $1$-cohomology groups of the mapping class groups. 
A new ingredient is to use the computation of a related twisted $1$-cohomology 
group by Morita.
The classification result implies in particular that there are no irreducible linear representations of 
dimension $2g+1$ for $g \geq 7$, which marks a contrast with the case $g=2$.
\end{abstract} 

\tableofcontents

\section{Introduction} \label{intro}\par

Linear representations are fundamental objects associated with the mapping class group of a surface.
However, the whole picture of them, even in low dimensions, are not understood very well.
Until the appearance of the Topological Quantum Field Theory around 1990, the symplectic representation 
was almost the only known linear representation, whereas the residually finiteness 
(\cite{grossman}, \cite{birman}) had assured the existence of other linear representations.
Even after the TQFT, the existence of lower dimensional representations which does not
factor through the permutation group of punctures continued to remain unknown. 
We note, if the surface has $r >1$ punctures, that the mapping class group usually considered has a 
surjection to the symmetric group of degree $r$ 
which is induced by the permutation of the punctures, 
and therefore the totality of linear representations should contain the representation theory of 
symmetric group and would be rather complicated.
\par

In 2011, Franks--Handel \cite{franks-handel} and then Korkmaz \cite{ldl} 
considered the {\em pure} version of the mapping class group of a surface of genus $g \geq 1$, 
the group of those mapping classes with trivial permutation of the punctures, 
and showed that there exist in fact no nontrivial complex linear representations of dimensions 
less than $2g$, with exceptions of abelian representations which occur only in the case of $g \leq 2$. 
This result was soon succeeded by Korkmaz \cite{sru}, who showed up to conjugation 
that the symplectic representation 
is the only nontrivial complex linear representation of dimension $2g$ in the case of $g \geq 3$.
\par

In this paper, following these works of Franks--Handel and Korkmaz, we consider 
the complex linear representations of dimension $2g+1$. As was studied by Trapp \cite{trapp}, such a 
representation can be constructed from a crossed homomorphism of the mapping class group with values 
in the $1$-homology group of the {\em closed} surface of the same genus. 
Our main result asserts, in the case of $g \geq 7$, 
that all nontrivial representations of dimension $2g+1$ can be obtained in this way up to dual
(Theorem \ref{surjectivity-thm} ). By using this main result, in this introduction, we give a 
complete classification of the $(2g+1)$-dimensional complex linear representations for $g \geq 7$. 
It follows in particular that there are no irreducible representations of dimension $2g+1$.
The classification is described in terms of certain twisted 
$1$-cohomology group of the (pure) mapping class group.
A new ingredient to prove the main result is to use the computation of a related twisted $1$-cohomology 
group by Morita \cite{morita_jacobian_I}.

\subsection{Statement of the classification result}\par

\subsection*{Setting}\par
Let $S=S_{g,r}^p$ denote a connected compact oriented surface of genus $g \geq 1$ with $p \geq 0$ 
boundary components and 
$r \geq 0$ punctures in interior. Here, a puncture is meant by a marked point.
The mapping class group of $S$, denoted by $\M{S}$, is defined to be
the group of the isotopy classes of the orientation preserving homeomorphisms of $S$ which preserve 
the marked points and the boundary, both pointwise.
Here, the isotopies are assumed to preserve the marked points and the boundary, both pointwise.
An $n$-dimensional linear representation of $\M{S}$ is simply a group homomorphism $\M{S} \to \GL{n, \C}$.
The two $n$-dimensional linear representations $\phi_1$ and  $\phi_2 : \M{S} \to \GL{n, \C}$  
are said to be conjugate if there exists 
some $A \in \GL{n, \C}$ such that $\phi_1(f) = A \phi_2(f) A^{-1}$ for each $f \in \M{S}$. Let 
$X$ be the set of conjugacy classes of all $(2g+1)$-dimensional linear representations of $\M{S}$. 
We denote by $X_0$ the set $X$ with the class of the trivial representation removed.
\par

\subsubsection*{Symplectic representation $\symp$}\par

Let $\Sb$ denote the closed surface obtained from $S$ by gluing a $2$-disk 
along each boundary component of $S$ 
and forgetting all the punctures.  We denote by $H$ the homology group $H_1(\Sb;\Z)$, which is 
a free abelian group of rank $2g$, and by $H_\C$ the homology group $H_1(\Sb;\C)$ with coefficients 
in $\C$. Note that $H_\C$ is canonically isomorphic to $H \otimes_\Z \C$ so that we may 
consider as $H \subset H_\C$. 
\par

The inclusion $S \hookrightarrow \Sb$ induces the homomorphism 
$\M{S} \to \M{\Sb}$ by extending mapping classes of $S$ with the identity on $\Sb \smallsetminus S$
so that the natural action 
of $\M{\Sb}$ on $H_\C$ induces a group homomorphism $\symp: \M{S} \to \GL{H_\C}$. 
We call $\symp$ the symplectic representation of $\M{S}$. We use the same symbol to denote the 
matrix form of this representation as $\symp: \M{S} \to \GL{2g,\C}$ with respect to a certain 
basis for $H_\C$, which we often do not specify explicitly. This would not make confusion since 
we are concerned with linear representations only up to conjugation. We consider $H_\C$, as well as
$\C^{2g}$, a left $\M{S}$-module via $\symp$ .
\par

\subsubsection*{Semidirect product and linear representation}\par

A {\em crossed homomorphism} of $\M{S}$, with values in a left $\M{S}$-module $M$, is 
a mapping $k: \M{S}  \to M$ which satisfies
\begin{equation}\label{crossed-hom}
 k(fg) = k(f) + f  k(g) \quad \text{for $f$, $g \in \M{S}$.} 
\end{equation} 
The crossed homomorphism $k$ is said to be {\em principal} if there exists some $m \in M$ such that 
$k(f) = fm - m$ for each $f \in \M{S}$. All the crossed homomorphisms of $\M{S}$ with values in $M$
naturally form an additive group, and 
its quotient by the subgroup consisting of all the principal 
crossed homomorphisms is isomorphic to the first cohomology group of $\M{S}$
with twisted coefficients in $M$, which we denote by $H^1(\M{S}; M)$.
\par
In order to construct a $(2g+1)$-dimensional linear representation from a crossed homomorphism, 
we consider the $\M{S}$-module $H_\C = \C^{2g}$, via the symplectic representation. 
For the natural left action of $\GL{2g,\C}$ on $\C^{2g}$, we denote the corresponding 
semidirect product by 
$\C^{2g} \rtimes \GL{2g, \C}$.
The correspondence 
$$ (z, A) \mapsto \left(\begin{array}{c|c} A & z  \\ \hline 0 & 1\end{array}\right) \quad 
    \text{($z \in \C^{2g}$, $A \in \GL{2g, \C}$)}$$
defines an injective homomorphism $i : \C^{2g} \rtimes \GL{2g, \C} \to \GL{2g+1,\C}$. Given a crossed 
homomorphism $c : \M{S} \to \C^{2g}$, with values in the left $\M{S}$-module $ H_\C = \C^{2g}$, the 
correspondence
$$ f \in \M{S} \mapsto (c(f), \symp(f))$$
defines a homomorphism $\M{S} \to \C^{2g} \rtimes \GL{2g, \C}$. Composing this homomorphism 
with $i$, one obtains a
nontrivial $(2g+1)$-dimensional linear representation 
$$\phi_c : \M{S} \to \GL{2g+1, \C}.$$
It will turn out that its conjugacy class $[\phi_c] \in X_0$ depends 
only on its cohomology class in $H^1(\M{S}; H_\C)$, and furthermore, does not change under  
scalar multiplication 
by a nonzero complex number in the cohomology group (Lemma \ref{well-definedness}). 
Therefore, the correspondence which sends the crossed homomorphism c 
to $[\phi_c] \in X_0$ descends to a mapping 
$$\sigma: H^1(\M{S}; H_{\C})/\C^{\times} \to X_0$$
where $\C^{\times}$ denotes $\C \smallsetminus \{ 0 \}$. 
\par

Now let $\iota : X_0 \to X_0$ be the involution which sends the conjugacy class of 
any linear representation $\phi: \M{S} \to \GL{2g+1, \C}$ to the class of its dual representation 
$$\phi^*: \M{S} \to \GL{2g+1,\C}, \quad 
\phi^*(f) = \left( \transpose{\phi(f)}\right)^{-1} \quad \text{for $f \in \M{S}$.}$$
Let $\overline{\sigma} : H^1(\M{S}; H_{\C})/\C^{\times} \to X_0/\langle \iota \rangle$ denote the 
composition of $\sigma$ with the quotient mapping $X_0 \to X_0/\langle \iota \rangle$.
We denote by $\Fix{\iota}$ the set of fixed points of $\iota$.
The following classification result states that we can identify 
$H^1(\M{S}; H_{\C})/\C^{\times}$ with almost the exact half of $X_0$ via $\sigma$,
if $g$ is sufficiently large:
\ppar
\begin{theorem} \label{clsf-thm}
 Let $g \geq 7$.
 \begin{enumerate}
  \item[(1)] The mapping $\overline{\sigma}: H^1(\M{S}; H_{\C})/\C^{\times} \to X_0/\langle \iota \rangle$ 
	     is bijective.
  \item[(2)] The set $\Fix{\iota}$ consists of a single element represented by the direct sum of the 
	     symplectic representation and the $1$-dimensional trivial representation.
 \end{enumerate}
\end{theorem}
\begin{remark}
 (1) As a consequence, for two crossed homomorphisms $c$ and $c^\prime$, we can easily see, 
 once a basis of $H_\C$ is fixed arbitrarily, that the representations $\phi_c$ and $\phi_{c^\prime}$ 
 represent the same class in $X_0$ if and only if 
 they are conjugate by an element in $i(\C^{2g} \rtimes \GL{2g, \C})$.
\par
 (2) For the case $r=0$ and $p=1$, the computation by Morita \cite{morita_jacobian_I} implies
 $$\#{H^1(\M{S};H_{\C})/\C^\times} = 2$$ ({\em c.f.} Theorem \ref{coeff_H_C}). 
 Therefore, at least for $g \geq 7$, 
 we can conclude that the two $(2g+1)$-dimensional 
 complex linear representations, the one constructed by Trapp \cite{trapp}, and the other 
 by Matsumoto--Nishino--Yano  \cite{matsumoto-nishiyama-yano} are conjugate up to dual. It might be 
 interesting to point out that both the representations seem to have their origins in
 Iwahori--Hecke algebra of Artin group.
\par
 (3) The above Theorem \ref{clsf-thm} does not hold for $(g,r,p) = (2,0,0)$. In fact, while 
 Morita's computation ({\em ibid})
  shows $H^1(\M{S};H_\C) = 0$ in this case, the Jones 
 representation of genus $2$ (\cite{jones}, \cite{kasahara}), a modification of an Iwahori--Hecke algebra 
 representation of the $6$-strand braid group,  gives a family of infinitely many non-conjugate irreducible
 $5$-dimensional complex linear representations of $\M{S}$. It is not known whether the theorem holds for 
 $g = 3$-$6$ or not.
\end{remark}
\ppar

The following is an obvious corollary to Theorem \ref{clsf-thm}:
\begin{corollary} \label{cor-obvious}
 Let $g \geq 7$. Then $\M{S}$ has no {\em irreducible} complex linear representations of dimension $2g+1$.
\end{corollary}
\ppar

Conversely, this corollary, together with results by Franks--Handel and Korkmaz, implies the 
surjectivity of $\overline{\sigma}$ immediately (Theorem \ref{surjectivity-thm}). 
However, our argument for the corollary 
simultaneously implies the surjectivity of $\overline{\sigma}$, the proof of which occupies the most part 
of this paper. It might be worthwhile to pursue simpler arguments for the corollary.
\ppar

\subsection{Proof of  Theorem \ref{clsf-thm}} \par

Suppose $g \geq 1$ for the moment. 
We first check that the mapping $\sigma$ is well-defined. Recall that for a crossed homomorphism
$c: \M{S} \to H_\C = \C^{2g}$ via the symplectic representation $\symp$, the linear representation 
$\phi_c : \M{S} \to \GL{2g+1, \C}$ is defined by 
\begin{equation}\label{formula_phi_c}
    \phi_c (f) = \left(\begin{array}{c|c} \symp(f) & c(f)  \\ \hline 0 & 1\end{array}\right)
    \quad \text{for $f \in \M{S}$.}
\end{equation}
\ppar

\begin{lemma} \label{well-definedness}
 \begin{enumerate}
  \item[(1)] The conjugacy class $[\phi_c] \in X_0$ depends only on the cohomology class 
	of $c$ in $H^1(\M{S}; H_\C)$.
  \item[(2)] For any $z \in \C^\times$, $[\phi_{zc}] = [\phi_c]$ in $X_0$.
 \end{enumerate}
\end{lemma}

\begin{proof}
 (1) Let $c^\prime : \M{S} \to \C^{2g}$ be another crossed homomorphism representing $[c] \in H^1(\M{S};\C)$. 
 We can then choose $w_0 \in \C^{2g}$ so that 
 $$ c(f) - c^\prime(f) = \symp(f)w_0 - w_0 $$
 for each $f \in \M{S}$. Putting $A := 
       \left(\begin{array}{c|c} I  & w_0  \\ \hline 0 & 1\end{array}\right)$ where $I$ denotes the 
 identity matrix, a direct computation implies 
 $$ A \phi_c(f) A^{-1} = \phi_{c^\prime}(f) \quad \text{for each $f \in \M{S}$.}$$
 This shows $[\phi_{c^\prime}] = [\phi_c]$.
 \par
 (2) For any $z \in \C^\times$, let $A := 
       \left(\begin{array}{c|c} zI  & 0  \\ \hline 0 & 1\end{array}\right)$. Then for each $f \in \M{S}$,
 a direct computation implies 
 $$ A \phi_c(f) A^{-1} = \phi_{zc}(f).$$
Hence we have $[\phi_{zc}] = [\phi_c]$ in $X_0$.
\end{proof}
\ppar

In view of this lemma, the following is well-defined.

\begin{definition}
 We define $\sigma: H^1(\M{S}; H_\C)/\C^{\times} \to X_0$ as the mapping which sends the class represented by 
 a crossed homomorphism $c: \M{S} \to H_\C = \C^{2g}$ to $[\phi_c] \in X_0$. We also define 
 $\bar{\sigma}: H^1(\M{S}; H_\C)/\C^{\times} \to X_0/\langle \iota \rangle$ as the composition of $\sigma$
 with the quotient mapping $X_0 \to X_0/\langle \iota \rangle$ where $\iota$ denotes the involution 
induced by taking the dual representation. 
\end{definition}
\par

We remark that $\sigma$, and hence $\bar{\sigma}$, are independent of the identification $H_\C = \C^{2g}$, 
{\em i.e.}, the choice of a basis of $H_\C$.
\ppar

To prove Theorem \ref{clsf-thm}, the most crucial is the following.

\begin{theorem} \label{surjectivity-thm}
 For $g \geq 7$, the mapping $\bar{\sigma}$ is surjective.
\end{theorem}

The assumption $g \geq 7$ for Theorem \ref{clsf-thm} is necessary to apply this theorem.
In particular, the injectivity of $\bar{\sigma}$ holds for $g \geq 1$ as shown below.
The proof of this theorem occupies most of the remaining sections. 
\ppar

We proceed to complete the proof of Theorem \ref{clsf-thm} assuming Theorem \ref{surjectivity-thm}. 
\ppar

Suppose $g \geq 1$ again. We first check $\bar{\sigma}$ is injective. 
Any element of $H^1(\M{S};H_\C)$ is represented by a 
crossed homomorphism $c: \M{S} \to H_\C = \C^{2g}$. We denote its representing class in $H^1(\M{S};H_\C)$ 
and $H^1(\M{S};H_\C)/\C^{\times}$ by $[c]$ and $\overline{[c]}$, respectively.
\par

Let $c_1$, $c_2 : \M{S} \to H_\C = \C^{2g}$ be two crossed homomorphisms which satisfy 
$\bar{\sigma}(\overline{[c_1]}) = \bar{\sigma}(\overline{[c_2]})$. We then have either $[\phi_{c_1}] = [\phi_{c_2}]$ or
$[\phi_{c_1}] = \iota([\phi_{c_2}])$.
\par

In the case $[\phi_{c_1}] = [\phi_{c_2}]$, choose $A \in \GL{2g+1, \C}$ 
so that $\phi_{c_2}(f) = A \phi_{c_1}(f) A^{-1}$ for each $f \in \M{S}$. This equality is equivalent to 
$\phi_{c_2}(f) A = A \phi_{c_1}(f)$, and if we take a block decomposition of $A$ as 
       $\left(\begin{array}{c|c} A_0 & w  \\ \hline s & x \end{array}\right)$ with $A_0$ a $2g \times 2g$ 
matrix and $x \in \C$, it becomes
\begin{equation} \label{cls-1}
 \left(\begin{array}{c|c} \symp(f)A_0 + c_2(f)s & \symp(f)w + c_2(f)x  \\ 
       \hline s & x \end{array}\right) \\
 = \left(\begin{array}{c|c} A_0 \symp(f) & A_0 c_1(f) + w  \\ 
	 \hline s \symp(f) & sc_1(f) + x \end{array}\right)
\end{equation}
for each $f \in \M{S}$.
In view of the lower left block, we have 
$$s = s \symp(f).$$
Namely, $s$ is a $\M{S}$-homomorphism of $H_\C$ to the trivial $\M{S}$-module $\C$. Therefore, 
by the irreducibility of the symplectic representation $\symp$, $s=0$ ({\em c.f.} \/ Remark 
\ref{symp-trivial}). 
We now have $A = \left(\begin{array}{c|c} A_0 & w  \\ \hline 0 & x \end{array}\right)$ and hence 
$A_0 \in \GL{2g, \C}$. Then the upper left block of \eqref{cls-1} becomes $\symp(f) A_0 = A_0 \symp(f)$, 
which means $A_0$ is a $\M{S}$-endomorphism of $H_\C$. Therefore, Schur's lemma and the irreducibility 
of $\symp$ imply $A_0 = z I$ for some $z \in \C^{\times}$. Now the upper right block of \eqref{cls-1} becomes
$$ \symp(f) w + c_2(f) x = z c_1(f) + w.$$
Hence, for each $f \in \M{S}$, we have 
$$c_1(f) = (x/z) c_2(f) + \symp(f) (w/z) - (w/z),$$
which shows $\overline{[c_1]} = \overline{[c_2]}$.
\par

In the case $[\phi_{c_1}] = \iota([\phi_{c_2}])$, the representation $\phi_{c_1}$ has both an invariant 
subspace of dimension $2g$ and an invariant subspace of dimension $1$, since the dual of $\phi_{c_2}$ has
an invariant $1$-dimensional subspace. 
Then the irreducibility of $\symp$ implies that $\phi_{c_1}$ is conjugate to 
the direct sum of $\symp$ and a $1$-dimensional linear representation $\varepsilon^\prime$, 
so that we have $[\phi_{c_1}] = [\symp \oplus \varepsilon^\prime]$. By considering the 
determinants of the representations on both sides of this formula, we see $\varepsilon^\prime$ 
is actually the trivial representation $\varepsilon$, since for each $f \in \M{S}$ it can be 
easily seen $\det{\symp(f)} = 1$,  and then by \eqref{formula_phi_c}, $\det{\phi_{c_1}(f)} = 1$.
Hence we have $[\phi_{c_1}] = [\symp \oplus \varepsilon] = [\phi_0]$. We then have $[c_1] = 0$ by the 
argument of the previous case.
Since $\iota$ is an involution of $X_0$, the same argument implies $[c_2] = 0$. In particular, 
we have $\overline{[c_1]} = \overline{[c_2]}$. This proves the injectivity of $\bar{\sigma}$ for $g \geq 1$.
\ppar

Next, we prove the second part of Theorem \ref{clsf-thm}. We clearly have 
$[\symp \oplus \varepsilon] \in \Fix{\iota}$, since both $\symp$ and $\varepsilon$ are self-dual.
Conversely, let $\phi : \M{S} \to \GL{2g+1, \C}$ be a representative of a fixed point of $\iota$, 
and now suppose $g \geq 7$. Then by Theorem \ref{surjectivity-thm}, 
We have 
$$ [\phi] = [\phi_c] = \iota([\phi_c])$$
for some crossed homomorphism $c: \M{S} \to H_{\C}$. Then the argument above implies again $[c] = 0$, 
and hence $[\phi] = [\rho_0 \oplus \varepsilon]$.
\par
This completes the proof of Theorem \ref{clsf-thm}, assuming Theorem \ref{surjectivity-thm}. \qed
\ppar

\subsection{Outline of paper} \par

The rest of this paper is essentially devoted to the proof of Theorem \ref{surjectivity-thm} and is 
organized as follows. In Section \ref{preliminaries}, we review some fundamental
results we need in later sections. We then start with the analysis of eigenvalue and eigenspace of the 
image $\phi(t_a)$ of the Dehn twist along any nonseparating simple closed curve $a$ under a nontrivial 
$(2g+1)$-dimensional linear representation $\phi : \M{S} \to \GL{2g+1, \C}$. In section \ref{eigenvalue},
we review and slightly generalize results of Korkmaz \cite{sru} to show that $\phi(t_a)$ has a unique 
eigenvalue $1$, and to provide a certain restriction of the dimension of the corresponding eigenspace.
In Section \ref{further_restriction}, we give a further restriction on the eigenspace 
with the assistance of a certain twisted $1$-cohomology group of the mapping class group of a surface 
and see that the 
Jordan form of $\phi(t_a)$ is uniquely determined. We then apply this result to a finite generating set 
of $\M{S}$, which consists of Dehn twists along nonseparating simple closed curves and is given in 
Theorem \ref{generators}, to complete the proof of Theorem \ref{surjectivity-thm}.
An algebraic characterization for images of the generators is presented in Section \ref{image_generators}, 
and is used to prove Theorem \ref{surjectivity-thm} in Section \ref{dichotomy}.
The characterization for the images consists of a previously known theorem by Korkmaz \cite{sru} 
(Theorem \ref{sru-lemma-4.7}) and our new Theorems \ref{prelim_thm:2g+1} and \ref{for_extra_gen}. 
The proofs of the 
latter two theorems, which depend only on the braid and commuting relations among the generators, require
rather long computational argument and are postponed to Section \ref{braid_commuting}. 
Section \ref{straightforward} provides a remark on a 
straightforward proof of Korkmaz's classification theorem \cite{sru} of $2g$-dimensional 
linear representations. 
Finally in Section \ref{appendix}, we discuss a generalization toward higher dimensional linear 
representations.
\par

\subsection*{Notation} \par
For a simple closed curve $a$ on $S$, the Dehn twist along $a$ is denoted by $t_a$.
By a Dehn twist, we always mean the right-handed one. For a linear representation 
$\phi: \M{S} \to \GL{2g+1,\C}$, we denote by $L_a$ the image of the Dehn twist
$\phi(t_a)$. For an eigenvalue $\lambda$ of $L_a$, its eigenspace is denoted by $E_\lambda^a$.
For a square matrix $M$, the multiplicity of $\lambda$ in the characteristic 
polynomial of $M$ is denoted by $\lambda_\#(M)$. We will omit the symbol $M$ and simply write
$\lambda_\#$ if what it means is clear from the context. 
\par

For square matrices $M_1$, $M_2$, \ldots, $M_k$, $\diag{M_1, M_2, \ldots, M_k}$ denotes 
the block diagonal matrix with block diagonals $M_1$, $M_2$, \ldots, $M_k$.
The identity matrix of degree $n$ is denoted by $I_n$.
We consider an element of $\C^m$ as a {\em column} \/ vector.
\ppar

\subsection*{Acknowledgements} \par
The author is grateful to Mustafa Korkmaz for helpful discussions, reading drafts of this paper, 
and kindly sending to the author the revised and combined version of \cite{ldl} and \cite{sru}.
He is also grateful to  Nariya Kawazumi for a comment on the treatment of the twisted $1$-cohomology group
of $\M{S_{g,r}^p}$ for general $(p,r)$, and to Masatoshi Sato for communicating his computation result for  
the twisted $1$-cohomology group.  This work was partially supported by JSPS KAKENHI Grant Number 
19K03498. The author is grateful to the anonymous referee for valuable comments and suggestions 
for clarification.
\par

 \section{Preliminaries} \label{preliminaries}
\par

In this section, we collect the previous results on or related to low dimensional 
complex linear representations of the mapping class group $\M{S}$. We also recall some
generality result on mapping class groups. More technical results will be recalled 
in later sections when necessary.

\subsection{Statements of Franks--Handel and Korkmaz's works} \par

We begin with the precise statement of the classification results by Franks--Handel and Korkmaz.
We note the first homology group $H_1(\M{S}; \Z)$ is the abelianization of the group $\M{S}$.

\begin{theorem}[Franks--Handel \cite{franks-handel} and Korkmaz \cite{ldl}]
 \label{fh-k} Let $g \geq 1$ and $n \leq 2g-1$. Then any linear representation $\phi: \M{S} \to \GL{n, \C}$
 factors through the abelianization map $\M{S} \to H_1(\M{S}; \Z)$.
 In particular,  if $g \geq 3$, then $\phi$ is trivial. ({\em c.f.}\, Theorem \ref{abelianization}.)
\end{theorem}
\par

\begin{theorem}[Korkmaz \cite{sru}] \label{2g-rep}
 Let $g \geq 3$. Then, any nontrivial representation $\phi: \M{S} \to \GL{2g, \C}$ is conjugate to the 
 symplectic representation $\rho_0$.
\end{theorem}

\subsection{Generalities on mapping class groups}\par

 We refer to Farb--Margalit \cite{farb-margalit}
as a basic reference.
\par

\begin{theorem}[\cite{korkmaz-mccarthy}, Theorem 1.2] \label{sru-theorem-3.1}
 Let $g \geq 1$, and let $a$ and $b$ be two nonseparating simple closed curves on $S$. Then, there exists 
 a sequence 
 $$a = a_1, a_2, \ldots, a_k = b$$
 of nonseparating simple closed curves on $S$ such that each consecutive pair $a_i$ and $a_{i+1}$ intersect 
 transversely at a single point.
\end{theorem}
\par

The first homology group $H_1(\M{S};\Z)$ is given by Korkmaz \cite{ldh} as follows.
\par

\begin{theorem} \label{abelianization}
 For $g \geq 1$, 
 \begin{equation*}
   H_1(\M{S};\Z) \cong 
    \begin{cases}
     \Z/12\Z & \text{if $(g, p) = (1, 0)$;} \\
     {\Z}^p & \text{if $g=1$ and $p \geq 1$;} \\
     \Z/10\Z & \text{if $g=2$;} \\
     0 & \text{otherwise.}
    \end{cases}
 \end{equation*}
\end{theorem}
\ppar

For a simple closed curve $a$ on $S$, the right-handed Dehn twist 
along $a$ is denoted by $t_a$. We collect here some useful relations among Dehn twists.

\begin{theorem} \label{dtwist-rel}
	\begin{enumerate}
	 \item[(1)] For any simple closed curve $a \subset S$ and $f \in \M{S}$, 
		    $$f t_a f^{-1} = t_{f(a)}.$$
	 \item[(2)] (commuting relation) For any two disjoint simple closed curves $a$ and $b \subset S$, 
		    $$t_a t_b = t_b t_a.$$
	 \item[(3)] (braid relation) If two simple closed curves $a$ and $b \subset S$ intersect 
		    transversely at a single point, 
		    $$ t_a t_b t_a = t_b t_a t_b.$$
	 \item[(4)] (lantern relation) Suppose $S = S_{0,0}^4$. Let $a$, $b$, $c$, $d$ be the boundary curves of $S$, and let $x$, $y$, $z$ be the simple closed curves on $S$ depicted in Figure \ref{lantern-fig}. Then 
			$$t_a t_b t_c t_d = t_x t_y t_z.$$
	\end{enumerate}
\end{theorem}
\par

\begin{figure}
		\centering
		\includegraphics[scale=0.65]{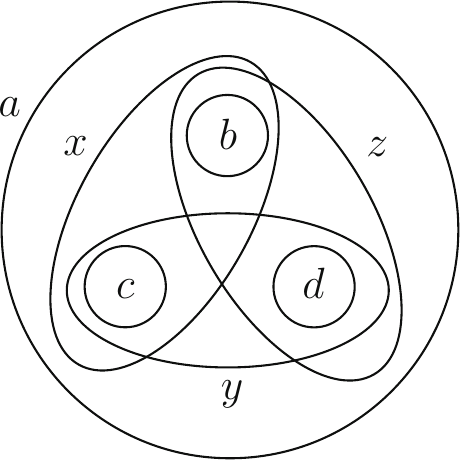}
		\caption{Lantern relation}
		\label{lantern-fig}
\end{figure}
\ppar

The following explicit generating set of $\M{S}$ can be derived from the Lickorish generators of $\M{\Sb}$
by repeated applications of the Birman exact sequence as well as the star relation, for instance 
({\em c.f.}\/ \cite{farb-margalit}):
\par

\begin{theorem} \label{generators}
	Let $g \geq 2$. The group $\M{S}$ is generated by the Dehn twists along the following 
	nonseparating simple closed curves depicted in Figures \ref{fig_gen_1} and \ref{fig_gen_2}:
	\begin{itemize}
		\item $a_1$, $b_1$, $a_2$, $b_2$, \ldots, $a_g$, $b_g$;
				 $c_1$, $c_2$, \ldots, $c_{g-1}$;
		\item $e_1$, $e_2$, \ldots, $e_p$;
				 $f_1$, $f_2$, \ldots, $f_r$.	
	\end{itemize}
\end{theorem}
\par

We remark that these generators are actually excessive, but are convenient for our purpose.
\par

\begin{figure}
		\centering
		\includegraphics[scale=1.0]{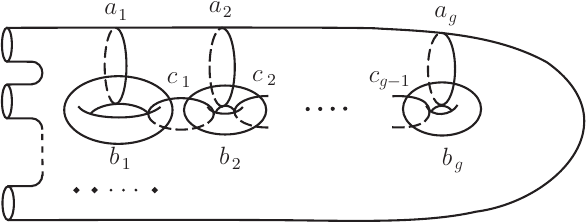}
		\caption{Explicit generators (1)}
		\label{fig_gen_1}
\end{figure}
\begin{figure}
		\centering
		\includegraphics[scale=1.0]{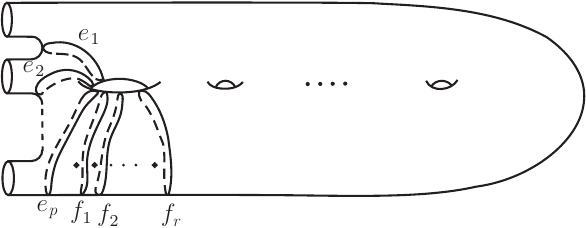}
		\caption{Explicit generators (2)}
		\label{fig_gen_2}
\end{figure}
\ppar

\subsection{Triviality of representation}\par 

The following triviality criterion for a linear representation was proved by Korkmaz \cite{sru} using 
Theorems \ref{fh-k} and \ref{abelianization} together with the nilpotency of the group of upper 
triangular unipotent matrices.

\begin{theorem}[\cite{sru}, Lemma 4.8] \label{triviality}
	Let $g \geq 3$, and $\psi: \M{S} \to \GL{m, \C}$ be an arbitrary homomorphism with 
	$m \geq 1$. Suppose there exists a flag 
	$$ 0 = W_0 \subset W_1 \subset W_2 \subset \cdots \subset W_k = \C^m$$
	which is $\M{S}$-invariant under the linear action $\psi$ such that 
	$\dim{W_{i+1}/W_i} \leq 2g-1$ for \/ $0 \leq i \leq k-1$. Then $\psi$ is trivial.
\end{theorem}

\begin{remark} \label{triviality-abelian}
 In case of $g=2$, Theorem \ref{triviality} does not hold true, but becomes true if the 
 consequence is changed to that the image of $\psi$ is an abelian group ({\em ibid}).
 This is due to the less known fact that the commutator subgroup of $\M{S}$ is perfect
 for $g \geq 2$ (\cite[Theorem 4.2]{korkmaz-mccarthy}).
\end{remark}

\begin{remark} \label{symp-trivial}
 Theorem \ref{triviality} together with Remark \ref{triviality-abelian} 
 implies immediately that the symplectic representation $\rho_0$ is irreducible, for $g \geq 2$.
\end{remark}

 \section{Eigenvalue and eigenspace} \label{eigenvalue}
\par

In this section, following the argument of Korkmaz \cite{sru}, we determine the eigenvalue 
of the image of the Dehn twist along a nonseparating simple closed curve under any $(2g+1)$-dimensional
representation, and give a lower bound for the dimension of the corresponding eigenspace.
\par
In general, it is known that Theorem \ref{abelianization} 
implies such an eigenvalue under {\em any dimensional} representation must be a root of unity 
if $g \geq 3$ ({\em cf}\/ \cite{aramayona-souto} and \cite{button}).
In low dimensional case, however, the eigenvalue suffers from further restriction while 
the additional assumption on $g$ is necessary.

We begin with recalling necessary results by Korkmaz.
\par

\subsection{General results from \cite{sru}} \par 

Let $\phi: \M{S} \to \GL{m, \C}$ be a homomorphism. For a simple closed curve $a \subset S$,
the image $\phi(t_a)$ is denoted by $L_a$. For an eigenvalue $\lambda$ of $L_a$, the 
eigenspace of $L_a$ corresponding to $\lambda$ is denoted by $E_\lambda^a$.

\begin{theorem}[\cite{ldl}, Lemma 4.2] \label{ldl-lemma-4.2}
		Let $a$, $b$, $c$, $d$ be nonseparating simple closed curves on $S$ such 
		that $f(c) = a$ and $f(d) = b$ for some $f \in \M{S}$. Suppose $\lambda$ is an
		eigenvalue of $L_a = \phi(t_a)$. Then $E_\lambda^a = E_\lambda^b$ if and only if
		$E_\lambda^c = E_\lambda^d$.
\end{theorem}

The proof follows from Theorem \ref{dtwist-rel} (1) and the basic fact that the eigenspace 
of the conjugation of a linear mapping by a linear isomorphism coincides with the image of 
the eigenspace of the linear mapping under the linear isomorphism.
\par 

The next result follows from Theorems \ref{sru-theorem-3.1} and \ref{ldl-lemma-4.2} together
with the fact that $\M{S}$ is generated by Dehn twists along nonseparating simple closed
curves if $g \geq 2$.

\begin{theorem}[\cite{sru}, Lemma 4.3] \label{sru-lemma-4.3}
		Let $g \geq 2$ and $\phi: \M{S} \to \GL{m, \C}$ be a homomorphism.
		Suppose that $a$ and $b$ are two nonseparating simple closed curves
		on $S$ which intersect transversely at a single point. If $E_\lambda^a = E_\lambda^b$ 
		for an eigenvalue $\lambda$ of $L_a = \phi(t_a)$, then $E_\lambda^a$ is invariant 
		under the $\M{S}$-action via $\phi$.
\end{theorem}

A careful analysis using Theorem \ref{fh-k} and some linear algebra implies the following.

\begin{theorem}[\cite{sru}, Lemma 4.5]  \label{sru-lemma-4.5}
	Let $g \geq 3$ and $\phi: \M{S} \to \GL{m, \C}$ be a homomorphism. Suppose
	$a$ is a nonseparating simple closed curve on $S$. If $\mu$ is an eigenvalue 
	of $L_a = \phi(t_a)$ with multiplicity $\mu_\# \leq 2g-3$, then $\mu = 1$ and 
	the dimension of $E_\mu^a$ coincides with $\mu_\#$.
\end{theorem}

This theorem immediately implies the following:

\begin{theorem}[{\em c.f.} \cite{sru}, Corollary 4.6]\label{sru-cor-4.6}
 Let $g \geq 3$ and $\phi: \M{S} \to \GL{m,\C}$ be a homomorphism. 
 If $m \leq 4g-4$, then $L_a = \phi(t_a)$ has at most two eigenvalues.
\end{theorem}

\subsection{The case of dimension $2g+1$} \par

We now consider $(2g+1)$-dimensional representations.
The following is an analogue of Korkmaz \cite[Lemma 5.1]{sru}.
\begin{lemma} \label{LemA}
 Let $g \geq 5$, and $\phi: \M{S} \to \GL{2g+1, \C}$ be an arbitrary homomorphism.
 Let $a$ be a nonseparating simple closed curve on $S$, and $\lambda$ be an eigenvalue of
 $L_a = \phi(t_a)$. 
 \par
 Let $\lambda_\#$ denote the multiplicity of $\lambda$ in the characteristic polynomial
 of $L_a$.
 If $\lambda_\# \geq 4$, then  $\Dm{\ela} \geq 2g-1$; in particular,  it holds
 $\lambda_\# \geq 2g-1$.
\end{lemma}
\par

\begin{proof}[Proof of Lemma \ref{LemA}]
 We choose a non-separating simple closed curve $b$ on $S$ which intersects with $a$ transversely 
 at a single point.
 We consider a regular neighbourhood of $a \cup b$ in the interior of $S$, 
 and denote the closure of its complement  in $S$ by $R$.
 Since $R$ and $S \smallsetminus R$ both have genus at least one, the inclusion map $R \hookrightarrow S$
 induces an injective homomorphism $\M{R} \to \M{S}$, and via this, we may consider 
$\M{R}$ as a subgroup of $\M{S}$ (see Paris--Rolfsen \cite{paris-rolfsen}).
\ppar

Now assume $\Dm{\ela} \leq 2g - 3$.  We set
$$
  W := \begin{cases}
	\ker{(L_a - \lambda I)^{\lambda_\#}} & \text{if $\lambda_\# \leq 2g-3$} \\
	\ker{(L_a - \lambda I)^{4}} & \text{if $\lambda_\# \geq 2g-2$ and $\Dm{\ela} =1$} \\
	\ker{(L_a - \lambda I)^{3}} & \text{if $\lambda_\# \geq 2g-2$ and $\Dm{\ela} =2$} \\
	\ker{(L_a - \lambda I)^{2}} & \text{if $\lambda_\# \geq 2g-2$ and $\Dm{\ela} =3$} \\
	\ela  & \text{if $\lambda_\# \geq 2g-2$ and $4 \leq \Dm{\ela} \leq 2g-3$} \\
       \end{cases}
$$
 Then, since any element of $\M{R}$ commutes with $t_a$, $W$ is $\M{R}$--invariant,  and since $g \geq 5$, its dimension satisfies
$$ 4 \leq \Dm{W} \leq 2(g-1)-1.$$
Since $g-1 \geq 3$, we may apply  Theorem \ref{triviality} for the $\M{R}$--invariant flag  
$0 \subset W \subset \C^{2g+1}$ to see $\phi(\M{R})$ is a trivial group.
Since $t_a$ is conjugate to a Dehn twist contained in $\M{R}$, we have $L_a = I$, which contradicts to 
the assumption $\Dm{\ela} \leq 2g-3$. We hence have $\Dm{\ela} \geq 2g-2$.
\ppar

Next, assume $\Dm{\ela} = 2g-2$. 
If $\ela \neq \elb$, we have $2g-5 \leq \Dm{\ela \cap \elb} \leq 2g-3$.
On the other hand, $\ela \cap \elb$ is clearly $\M{R}$-invariant. We then 
consider the  $\M{R}$-invariant flag 
$$ 0 \subset \ela \cap \elb \subset \C^{2g+1}.$$
Since $g \geq 5$, we have $\Dm{\ela \cap \elb} \leq 2g-3$ and
$\Dm{\C^{2g+1}/\ela \cap \elb} \leq 6 \leq 2g-3$. Therefore,  Theorem \ref{triviality} implies 
$\phi(\M{R})$ is trivial, which contradicts to $\Dm{\ela} = 2g-2$. 
\par

If $\ela = \elb$, then $\ela$ is $\M{S}$-invariant by Theorem \ref{sru-lemma-4.3}. We can therefore 
 apply Theorem \ref{triviality} to the flag $0 \subset \ela \subset \C^{2g+1}$ to see 
 $\phi$ is trivial. This contradicts the assumption $\Dm{\ela} = 2g-2$.

\ppar
This proves $\Dm{\ela} \geq 2g-1$.
\end{proof}
\ppar

The following is a slight generalization of Korkmaz \cite[Lemma 5.2]{sru}, and tells
that an eigenvalue of $L_a$ must be $1$ if its eigenspace has a small codimension:

\begin{lemma} \label{LemB}
 Let $\phi: \M{S} \to \GL{m,\C}$ be an arbitrary homomorphism.
 Let  $a$ be a nonseparating simple closed curve on 
 $S$, and $\lambda$ be an eigenvalue of $L_a = \phi(t_a)$. 
 Let $s = m - \Dm{\ela}$ where $\ela$ denotes the corresponding eigenspace. 
 \par

 If $g \geq 3$ and  $m > 7s$, then $\lambda = 1$.
\end{lemma}

\begin{proof}
 Since $g \geq 3$, we can apply Theorem \ref{dtwist-rel} (4) to 
 choose nonseparating simple closed curves
 $c_1 = a$, $c_2$, \ldots, $c_7$ on $S$ so that 
 they satisfy the lantern relation 
 \begin{equation} \label{LemB_lantern}
  t_{c_1} t_{c_2} t_{c_3} t_{c_4} = t_{c_5} t_{c_6} t_{c_7}. 
 \end{equation}
 For each $i = 1$, $2$, \ldots, $7$, we can choose $\psi_i : \C^m \to \C^s$ so that
 $\ker{\psi_i} = \eigensp{\lambda}{c_i}$ with $\rank{\psi_i} = s$.
 Then we see $\bigcap_{i=1}^{7}{\eigensp{\lambda}{c_i}} = 
          \ker{\left(\Psi := \oplus_{i=1}^{7}{\psi_i}: \C^m \to \oplus_{i=1}^{7}{\C^s}\right)}$.
 Therefore, we have 
  $$\Dm{\bigcap_{i=1}^{7}{\eigensp{\lambda}{c_i}} } = m - \rank{\Psi} \geq m - 7s > 0.$$ 
 We may thus choose $v_0 \in \bigcap_{i=1}^7{\eigensp{\lambda}{c_i}}$ with $v_0 \neq 0$.
 We multiply the images under $\phi$ of the left- and right-hand sides of \eqref{LemB_lantern} with 
 $v_0$ to obtain
 $$\lambda^4 v_0 = \lambda^3 v_0.$$
 Since $L_a$ is nonsingular, we have $\lambda = 1$.
\end{proof}
\ppar

Now we are ready to prove the following:
\begin{theorem} \label{eigen_2g+1}
 Suppose $g \geq 7$. Let $a$ be a non-separating simple closed curve on $S$.
 For an arbitrary homomorphism 
 $\phi: \M{S} \to \GL{2g+1,\C}$, the image $L_a = \phi(t_a)$ of the Dehn twist along $a$
 has exactly one eigenvalue $1$. Furthermore, the eigenspace $\eigensp{1}{a}$ has dimension
 at least $2g-1$.
\end{theorem}
\par

\begin{proof}
 We first observe the eigenvalue of $L_a$ is unique. Since $g \geq 3$ and  $2g+1 \leq 4g-4$,
 The application of  Theorem \ref{sru-cor-4.6} implies the number of eigenvalues of $L_a$ is
 at most two.  Assume that $L_a$ has distinct eigenvalues $\mu$ and  $\lambda$ with
 $\mu_{\#} \leq \lambda_{\#}$. Here, $\mu_{\#}$ and $\lambda_{\#}$ denote the multiplicities
 of $\mu$ and $\lambda$, respectively, in the characteristic polynomial of $L_a$.
 Since $\mu_{\#} + \lambda_{\#} = 2g+1$, we see 
 $\mu_{\#} \leq g$ and $\lambda_{\#} \geq g+1$. In particular, since $\mu_{\#} \leq 2g-3$,
 by Theorem \ref{sru-lemma-4.5}, we have
 $\mu = 1$ and $\Dm{\ea{\mu}} = \mu_{\#}$.
 \par
 On the other hand, we see $\lambda_{\#} \geq 4$ since $\lambda_{\#} \geq g+1$, and also $g \geq 5$ by
 assumption. Therefore, Lemma \ref{LemA} implies
 $$\Dm{\ela} \geq 2g-1.$$
 This shows that the codimension $s = 2g+1 - \Dm{\ela}$ of $\ela$ satisfies 
 $s \leq 2$. Together with the assumption $g \geq 7$, we have
 $$ 2g+1 \geq 15 > 14 \geq 7s.$$
 Hence, we have $2g+1 > 7s$, and therefore Lemma \ref{LemB} implies $\lambda = 1$, which contradicts to
 $\lambda \neq \mu$. This proves the uniqueness of the eigenvalue of $L_a$.
\par

 Now, let $\lambda$ be the unique eigenvalue of $L_a$. Since $\lambda_{\#} = 2g+1 \geq 4$, 
 Lemma \ref{LemA} implies $\Dm{\ela} \geq 2g-1$ again, we hence have $\lambda = 1$ by
 Lemma \ref{LemB}. This completes the proof.
\end{proof}
\ppar

 \section{Further restriction of eigenspace} \label{further_restriction} \par

In this section, we show that the twisted cohomology group $H^1(\M{S_{g,0}^1}; H_{\C})$ 
gives a further restriction of $\Dm{\ea{1}}$ in Theorem \ref{eigen_2g+1}. We first review
necessary computation results. For generalities on cohomology of groups, we refer to 
\cite{brown}.

\par

\subsection{Twisted $1$-cohomology} \par

Recall that $S = S_{g,r}^p$ denotes the oriented surface of genus $g$ with $p$ boundary components
and $r$ punctures. We denote $H = H_1(\Sb;\Z)$ and $H_{\C} = H_1(\Sb; \C)$ where $\Sb$ denotes the closed 
surface of genus $g$ obtained from $S$  by by gluing a $2$-disk to each boundary component and 
forgetting the punctures. The homology group $H_1(\Sb)$ with coefficients in any abelian group is 
naturally a left $\M{S}$-module via the natural surjection $\M{S} \to \M{\Sb}$. In this setting, 
we first describe $H^1(\M{S};H_\C)$ for general $p$, $r$.

\begin{theorem} \label{coeff_H_C}
 Let $g \geq 2$. For $p$, $r \geq 0$, 
 \begin{equation*}
  H^1(\M{S}; H_\C) \cong \C^{p+r}
 \end{equation*}
 where $\C^{0}$ denotes $0$.
\end{theorem}

This theorem is due to Morita \cite{morita_jacobian_I} for the case $p+r \leq 1$, 
and is due to Hain \cite[Proposition 5.2]{hain} for the case $g \geq 3$ and general $p$, $r \geq 0$.
See also Putman \cite[Theorem 3.2]{putman} for the case $g \geq 3$. 
The computation with the full generality on $(g,p,r)$ was communicated to the author
by Sato \cite{sato}. 
\par

We remark that the twisted coefficients used by Morita and Hain in their computations are 
slightly different from $H_\C$: Morita computed $H^1(\M{S};H) \cong \Z^{p+r}$ and Hain computed
$H^1(\M{S};H_1(\Sb;\Q)) \cong \Q^{p+r}$ in the indicated range of $(g,p,r)$. The cohomology with 
coefficients in $H_\C$ can be obtained from these computations as follows.
\par

For any $g \geq 2$ and $p$, $r \geq 0$, the homology group $H_0(\M{S};H)$ is by definition 
the coinvariant of the $\M{S}$-module $H$, and can be easily seen to be zero. Therefore, a 
standard argument of the universal coefficient theorem implies, for $k=\Z$, $\Q$, $\C$,
\begin{equation*}
 H^1(\M{S}; H_1(\Sb;k)) \cong \Hom_\Z(H_1(\M{S};H),k).
\end{equation*}
Since $\M{S}$ is finitely generated, so is $H_1(\M{S};H)$. We hence have
\begin{align*}
 H^1(\M{S};H_\C) & \cong H^1(\M{S};H) \otimes_\Z \C \\
                 &  \cong H^1(\M{S}; H_1(\Sb;\Q)) \otimes_\Q \C,
\end{align*}
which implies Theorem \ref{coeff_H_C} except for the case $g=2$ and $p+r > 1$.
\ppar

In addition to the computation mentioned above, Morita combinatorially defined a crossed homomorphism 
$$k_0: \M{S_{g,0}^1} \to H$$
which represents a generator of $H^1(\M{S_{g,0}^1};H)$. This crossed homomorphism has the property 
$k_0(t_l) = 0$ for a certain nonseparating simple closed curve $l$ on $S_{g,0}^1$ 
(see \cite{morita}, especially the proof of Proposition 6.15 therein). This particular property 
implies the following.

\begin{theorem} \label{crossed_hom}
 Let $R$ be a compact connected oriented surface of genus at least $2$ with nonempty connected boundary
 and no punctures.
 Also let $\Rb$ denote the closed surface obtained from $R$ by gluing a $2$-disk along the boundary.
 Suppose 
  $$c: \M{R} \to H_1(\Rb;\C)$$
 is an arbitrary crossed homomorphism. Let $d$ be a nonseparating simple closed curve on $R$.
 Let $\tilde{d}$ denote the oriented curve obtained from d by choosing an arbitrary
 orientation, and $[\tilde{d}]$ denote its representing class in $H_1(\Rb;\C)$.
 Then
 $$c(t_d) = z [\tilde{d}]$$
 for some $z \in \C$.
\end{theorem}
\par

\begin{proof}
 Let
 $k_0:\M{R} \to H_1(\Rb;\Z)$ be Morita's crossed homomorphism above for $R$.
 We may consider $k_0$ represents a generator of $H^1(\M{R}; H_1(\Rb;\C)) \cong \C$ 
 via the inclusion  $H_1(\Rb;\Z) \hookrightarrow H_1(\Rb;\C) \cong H_1(\Rb;\Z) \otimes_{\Z} \C$.
 Hence, there exist $A \in \C$ and $x \in H_1(\Rb;\C)$ such that 
 \begin{equation}\label{c(f)}
  c(f) = A k_0(f) + f_*x -x
 \end{equation}
 for each $f \in \M{R}$. As mentioned above, there exists a nonseparating simple
 closed curve $l$ on $R$ such that $k_0(t_l) = 0$. By the classification theorem for surfaces, there exists
 $\varphi \in \M{R}$ such that $\varphi(l) = d$. Then we have $t_d = \varphi t_l \varphi^{-1}$ by Theorem
 \ref{dtwist-rel}.
 On the other hand, the property of crossed homomorphism \eqref{crossed-hom} implies
 $k_0(f^{-1}) = -f_*^{-1} k_0(f)$ for any $f \in \M{R}$, and hence
 \begin{align*}
    k_0(t_d)  & =  k_0(\varphi t_l \varphi^{-1}) = k_0(\varphi) + \varphi_* k_0(t_l \varphi^{-1}) \\
     & = k_0(\varphi) + \varphi_*(k_0(t_l) + (t_l)_* k_0(\varphi^{-1})) \\
     & = k_0(\varphi) + (\varphi t_l)_*k_0(\varphi^{-1})
    = k_0(\varphi)   - (\varphi t_l \varphi^{-1})_*k_0(\varphi) \\
    	         & =  k_0(\varphi) - (t_d)_*k_0(\varphi). 
 \end{align*}
In view of \eqref{c(f)}, we then have
 $$c(t_d) = (t_d)_* ( x - A k_0(\varphi) ) - (x - A k_0(\varphi) ). $$
As is well-known, the action of $t_d$ on $H_1(\Rb;\C)$ is given by 
 \begin{align} \label{P-L}
  (t_d)_*x = x + \langle [\tilde{d}],x\rangle [\tilde{d}] \quad\text{for $x \in H_1(\Rb;\C)$}
 \end{align}

where $\langle \cdot, \cdot\rangle$ denotes the algebraic intersection form on $H_1(\Rb;\C)$.
We then conclude
 $$c(t_d) = \langle [\tilde{d}], x - A k_0(\varphi)\rangle [\tilde{d}]
      = z [\tilde{d}]$$
with $z = \langle [\tilde{d}], x - A k_0(\varphi)\rangle \in \C$.
This completes the proof of Theorem \ref{crossed_hom}.
\end{proof}

\subsection{The restriction of $\Dm{\ea{1}}$}  \label{restriction} \par

We can now prove the following.
\begin{theorem} \label{eigen_dim}
 Suppose $g \geq 7$. Let $a$ be a nonseparating simple closed curve on $S$.
 For any non-trivial homomorphism $\phi: \M{S} \to \GL{2g+1, \C}$, the image 
 $L_a = \phi(t_a)$ has a unique eigenvalue $1$, and $\Dm{\ea{1}} = 2g$.
\end{theorem}

\begin{proof}
 By Theorem \ref{eigen_2g+1}, the matrix $L_a$ has a unique eigenvalue $1$, and it holds
 $\Dm{\ea{1}} \geq 2g-1$. If $\Dm{\ea{1}} = 2g+1$, 
 then the homomorphism $\phi$ is trivial,  since $L_a = I$ and $\M{S}$ is generated by 
 conjugations of $t_a$. Therefore, we have only to prove $\Dm{\ea{1}} \neq 2g-1$.
 Suppose to the contrary that $\Dm{\ea{1}} = 2g-1$.
 Let $b$ be a non-separating simple closed curve which intersects $a$ transversely at a
 single point. We divide into two cases according to whether $\ea{1}$ coincides with $\eb{1}$ or not.
\ppar

{\bf (I) The case $\ea{1} = \eb{1}$.}
 By Theorem \ref{sru-lemma-4.3}, 
 $\ea{1}$ is a $\M{S}$-invariant $(2g-1)$-dimensional subspace of $\C^{2g+1}$.
 Therefore, we have a $\M{S}$-invariant flag $0 \subset \ela \subset \C^{2g+1}$, whose all
 successive quotients have dimensions less than or equal to $2g-1$. Therefore, since $g \geq 3$,
 Theorem \ref{triviality} implies $\phi$ is trivial, which contradicts to $\Dm{\ea{1}} = 2g-1$.
 \ppar

{\bf (II) The case $\ea{1} \neq \eb{1}$.}
 We choose a compact connected subsurface $R$ of $S$ so that $R$ is disjoint from 
 $a \cup b$ and has genus $g-1$,  nonempty connected boundary, and no punctures, as 
 schematically depicted in Figure \ref{fig_R}. We may consider $\M{R}$ as a subgroup of $\M{S}$
 ({\em c.f.}\/ \cite{paris-rolfsen}).
  \begin{figure}
   \centering
   \includegraphics[scale=1.0]{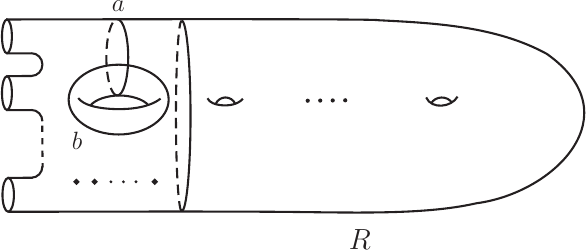}
   \caption{The subsurface $R$}\label{fig_R}
 \end{figure}
 Since the Dehn twists $t_a$ and $t_b$ commute with any element of $\M{R}$, both $\ea{1}$ and
 $\eb{1}$ are $\M{R}$-invariant, and so is $\ea{1} \cap \eb{1}$. Its dimension satisfies
 $$2g-3 \leq \Dm{\ea{1} \cap \eb{1}} \leq 2g-2.$$
 We divide into two subcases.
\ppar

{\bf (i) The case $\Dm{\ea{1} \cap \eb{1}} = 2g-3$.}
In the $\M{R}$-invariant flag  $0 \subset \ea{1} \cap \eb{1} \subset \C^{2g+1}$, we see
the dimensions of the successive quotients are less than or equal to $2(g-1)-1$, and 
$g-1 \geq 3$. Therefore, the restriction of $\phi$ to $\M{R}$ is trivial by Theorem \ref{triviality}.
Since $t_a$ is conjugate to an element of $\M{R}$, we have $L_a = \phi(t_a) = I$. This contradicts to
the assumption $\Dm{\ea{1}} = 2g-1$.
\ppar

{\bf (ii) The case $\Dm{\ea{1} \cap \eb{1}} = 2g-2$:}

In the $\M{R}$-invariant flag  $0 \subset \ea{1} \cap \eb{1} \subset \C^{2g+1}$, we see
$\Dm{\C^{2g+1}/\ea{1} \cap \eb{1}} = 3 \leq 2g-3$, and hence the induced action of $\M{R}$ on
$\C^{2g+1}/(\ea{1} \cap \eb{1})$ is trivial by Theorem \ref{fh-k}. On the other hand, since $g-1 \geq 3$, Theorem \ref{2g-rep}
implies that the action of $\M{R}$ on $\ea{1} \cap \eb{1}$ via $\phi$ is either A) trivial, or B)
conjugate to the symplectic representation of $\M{R}$, which we denote by
$\rho^R_0: \M{R} \to \GL{H_1(\bar{R};\C)}$. Here, we denote by $\bar{R}$ the closed surface
obtained from $R$ by gluing a $2$-disk along the boundary of $R$.

\par
In case A), we choose an arbitrary basis of $\ea{1} \cap \eb{1}$ and extend it arbitrarily
to a basis $\alpha$ of $\C^{2g+1}$. For each $f \in \M{R}$, we have
 $$\phi(f) = \left( \begin{array}{c|c}
		       I_{2g-2} &  * \\
		       \hline
			0 & I_{3}
		      \end{array}
		     \right)$$
according to $\alpha$. This implies that $\phi(\M{R})$ is an abelian group. Therefore, 
$\phi(\M{R})$ is trivial by Theorem \ref{abelianization}. As before, this implies
$L_a = \phi(t_a) = I$ and contradicts to $\Dm{\ea{1}} = 2g-1$.
\ppar

In case B), we may choose an isomorphism $u: \ea{1} \cap \eb{1} \to H_1(\Rb;\C)$
such that 
$$u(\phi(f) v) = f_* u(v) \quad \text{($f \in \M{R}$, $v \in \ela \cap \elb$)}.$$
Here, $f_*$ denotes the natural action of $f$ on $H_1(\Rb;\C)$.
\par

\ppar

We now fix a basis of $\C^{2g+1}$ extending an appropriate basis of $\ea{1} \cap \eb{1}$.
Then, under the identification of $\ea{1} \cap \eb{1}$ with $H_1(\Rb;\C)$ via $u$, the image of
$f \in \M{R}$ under $\phi$ has the form
$$\phi(f) = \left( \begin{array}{c|c}
		       \rho^R_0(f) &  \begin{matrix} w_1 & w_2 & w_3 \end{matrix} \\
		       \hline
			0 & I_{3}
		      	\end{array}\right) \qquad \text{($w_1$, $w_2$, $w_3 \in H_1(\Rb;\C)$)}.$$
For another $f^\prime \in \M{R}$ with
$$\phi(f^\prime) = \left( \begin{array}{c|c}
	       \rho^R_0(f^\prime) &  \begin{matrix} w_1^\prime & w_2^\prime & w_3^\prime \end{matrix} \\
	       \hline
	       0 & I_{3}
	      \end{array}\right) \qquad \text{($w_1^\prime$, $w_2^\prime$, $w_3^\prime
 		                                          \in H_1(\Rb;\C)$)},$$
we have
$$\phi(f f^\prime) = \left( \begin{array}{c|c}
		      \rho^R_0(f f^\prime) &  \begin{matrix} w_1 + f_* w_1^\prime &  w_2 + f_* w_2^\prime & 
					      w_3 + f_* w_3^\prime \end{matrix} \\
		       \hline
		       0 & I_3
		    \end{array}\right).$$
This formula shows for $i=1$, $2$, and $3$, that the correspondence
$f \in \M{R} \mapsto w_i$ defines a crossed homomorphism 
$$c_i: \M{R} \to H_1(\Rb;\C).$$
\par

Now let $d$ be a non-separating simple closed curve on $R$. We fix an orientation of
$d$ and denote its representing homology class by $[\tilde{d}] \in H_1(\Rb;\C)$.
Then, by Theorem \ref{crossed_hom}, there exists a complex number $z_i$ for each $i$ such that
$c_i(t_d) = z_i [\tilde{d}]$. On the other hand, the action of
$t_d$ on $H_1(\Rb;\C)$ is given by \eqref{P-L}.
Consequently, we have $\rank{(\phi(t_d) - I)}= 1$. Since $t_a$ is conjugate to $t_d$ in $\M{S}$,
we may conclude $\rank{(L_a - I)} = 1$, which contradicts the assumption $\Dm{\ea{1}}= 2g-1$.
\par
We may now conclude $\Dm{\ea{1}} \neq 2g-1$. This completes the proof of Theorem \ref{eigen_dim}.
\end{proof}

 \section{The images of generators of $\M{S}$}  \label{image_generators} \par

Let $\phi: \M{S} \to \GL{2g+1, \C}$ be any nontrivial homomorphism and $a$ be any nonseparating simple
closed curve on $S$. Theorem \ref{eigen_dim} in the previous section shows the uniqueness of the Jordan
form of $\phi(t_a)$ for $g \geq 7$. To prove our goal Theorem \ref{surjectivity-thm}, we have only to 
show that this restriction
on $\phi(t_a)$,  together with a certain nontriviality assumption, forces the image of a set of generators
of $\M{S}$ into the expected forms, with respect to some basis of 
$\C^{2g+1}$. As a set of generators, we use the one given by Theorem \ref{generators}.
\par

Among the generators of $\M{S}$, the Dehn twists along the curves
 \begin{itemize}
  \item $a_1$, $b_1$, $a_2$, $b_2$, \ldots, $a_g$, $b_g$;
	$c_1$, $c_2$, \ldots, $c_{g-1}$;
  \item $e_1$, $e_2$, \ldots, $e_p$;
	$f_1$, $f_2$, \ldots, $f_r$
 \end{itemize}
in Figures \ref{fig_gen_1} and \ref{fig_gen_2}, the images of $t_{a_i}$s' and $t_{b_i}$s' have
already been considered by Korkmaz. To state his result, we need to set up some notation.
\ppar

We use the symbol $U$ and $\Uh$ to denote the $2 \times 2$ matrices:
\begin{equation} \label{notation_U}
  U = \begin{pmatrix} 1 & 1 \\ 0 & 1  \end{pmatrix}, \quad
	\Uh = \begin{pmatrix} 1 & 0 \\ -1 & 1 \end{pmatrix}.
\end{equation}

\par 

\begin{definition} \label{A_and_B}
	For $i = 1$, $2$, \ldots, $g$, we define the $2g \times 2g$ matrices $A_i$ and $B_i$
	to be the block diagonal matrices
	\begin{align*}
		A_i & = \diag{I_2, I_2, \ldots, U, I_2, \ldots, I_2}, \\
		B_i & = \diag{I_2, I_2, \ldots, \Uh, I_2, \ldots, I_2},
	\end{align*}
	where $U$ and $\Uh$ are in the $i$th entry.
\end{definition}

We note that $A_i$ and $B_i$ are the images of the Dehn twists $t_{a_i}$ and $t_{b_i}$,
respectively under the symplectic representation 
$\rho_0$ with respect to the following basis $\{x_i, y_i \}$ of $H_1(\bar{S}; \C)$:

\begin{definition} \label{symp-basis}
	We define the basis $\{x_1, y_1, \ldots, x_g, y_g \}$ of 
	$H_\C = H_1(\bar{S};\C)$ as follows. Let $x_i$ and $y_i$ be the oriented curves on 
	$S$ depicted in Figure \ref{homology-basis}. For each $i$, the homology classes $x_i$ and $y_i$
	in $H_\C$ are defined to be the classes represented by the images of the oriented 
	curves denoted by the same symbols under the inclusion $S \hookrightarrow \bar{S}$.
\end{definition}
\begin{figure}
		\centering
		\includegraphics[scale=1.0]{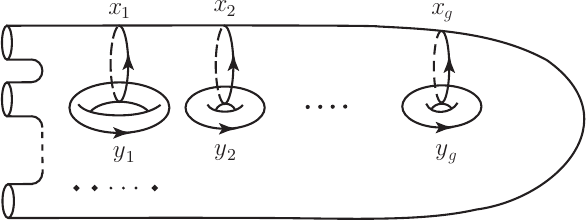}
		\caption{A basis of $H_\C$}
		\label{homology-basis}
\end{figure}

Now the result of Korkmaz can be stated as follows.

\begin{theorem}[\cite{sru}, Lemma 4.7] \label{sru-lemma-4.7}
	Let $g \geq 1$, $m \geq 2g$ and let $\phi: \M{S} \to \GL{m, \C}$ be a homomorphism.
	Let $a$ be any nonseparating simple closed curve on $S$. Suppose that the Jordan form
	of $L_a = \phi(t_a)$ is given by 
	$\left(\begin{array}{cc} U & 0 \\ 0 & I_{m-2}\end{array}\right)$.
	Suppose also that there exits a nonseparating simple closed curve $b \subset S$
	intersecting $a$ transversely at a single point such that $E_1^a \neq E_1^b$.
	Then there exists a basis of $\C^m$ with respect to which 
	$$
		L_{a_i} = \left(\begin{array}{cc}
					A_i & 0 \\
					0   & I_{m-2g}
				  \end{array}\right) 
		\quad \text{and} \quad
		L_{b_i} = \left(\begin{array}{cc}
					B_i & 0 \\
					0   & I_{m-2g}
				  \end{array}\right)
	$$
	for $i=1$, $2$, \ldots, $g$.
\end{theorem}
\ppar

Next, we describe two theorems which will control the images of the remaining generators of $\M{S}$
under any nontrivial {\em $(2g+1)$-dimensional} representation. We need to set up some further 
notation.

\begin{definition} \label{L_and_Ck}
 We set 
 $$ L = \begin{pmatrix} 1 &  1 & 0 & -1 \\ 
	               0 &  1 & 0 &  0 \\
	               0 & -1 & 1 &  1 \\
	               0 &  0 & 0 &  1
	\end{pmatrix}$$
 and 
 $$C_k = \diag{I_{2k-2}, L, I_{2g-2k-2}} \in \GL{2g, \C}$$
 for  $k =1$, $2$, \ldots, $g-1$.
\end{definition}

We note $C_k$ is the image of $t_{c_k}$ under the symplectic representation $\rho_0$
with respect to the basis $\{ x_i, y_i \}$.
\par

The following theorem will provide the control of the images of $t_{c_k}$s'.

\begin{theorem} \label{prelim_thm:2g+1}
 Let $g \geq 2$, and let 
  $$\At_i = \begin{pmatrix} A_i & 0 \\ 0 & 1 \end{pmatrix}, \, 
 \Bt_i = \begin{pmatrix} B_i & 0 \\ 0 & 1 \end{pmatrix} \in \GL{2g+1,\C}$$
 for $1 \leq i \leq g$. For each $k$ with $1 \leq k \leq g-1$, 
 suppose $\Xt_k \in \GL{2g+1, \C}$ satisfies the following conditions (i)--(iv):
\begin{enumerate}
 \item[(i)] $\Xt_k$ has exactly one eigenvalue $1$.
 \item[(ii)] $\Xt_k \At_i = \At_i \Xt_k$ for $i=1$,$2$, \ldots, $g$.
 \item[(iii)] If $g \geq 3$, then $\Xt_k \Bt_j = \Bt_j \Xt_k$ for each $j$ satisfying $1 \leq j \leq g$ with
	      $j \neq k$, $k+1$. 
 \item[(iv)] $\Xt_k \Bt_j \Xt_k = \Bt_j \Xt_k \Bt_j$ for $j=k$, $k+1$.
\end{enumerate}
 Then, there exist nonzero complex numbers $p_1$, $p_2$, \ldots, $p_{g-1}$ such that
 for  
 $$P=\diag{I_2, p_1 I_2, p_2 I_2, \ldots,p_{g-1} I_2} \in \GL{2g,\C} \quad \text{and}\quad 
 \Pt = \begin{pmatrix}P & 0 \\ 0 & 1 \end{pmatrix},$$  it holds
 $\Pt^{-1} \At_i \Pt = \At_i$ and $\Pt^{-1} \Bt_i \Pt = \Bt_i$ for each $i=1$, $2$, \ldots $g$, 
 and furthermore, it also holds
 for each $k=1$, $2$, \ldots, $g-1$, 
\begin{equation} \label{conj_xtk}
 \Pt^{-1} \Xt_k \Pt = \begin{pmatrix} C_k & \vect{w_k} \\ \transpose{\vect{s_k}} & 1 \end{pmatrix}
\end{equation}
 where $\vect{w_k}$, $\vect{s_k} \in \C^{2g}$ with either $\vect{w_k} = \vect{0}$ or $\vect{s_k} = \vect{0}$.
\end{theorem}
\par
\begin{remark}
 Conversely, if each $\Xt_k$ is given as the right-hand side of  \eqref{conj_xtk}, 
 it is easy to see the conditions (i)-(iv) are satisfied.
\end{remark}
\ppar

The next theorem will provide the control of the images of the rest of the generators.

\begin{theorem} \label{for_extra_gen}
 Let $g \geq 2$, and $\At_i$, $\Bt_i$ as in Theorem \ref{prelim_thm:2g+1}. 
 Suppose the matrix $\Ft \in \GL{2g+1,\C}$
 satisfies the following conditions (i)-(iv):
\begin{enumerate}
 \item[(i)] $\Ft$ has exactly one eigenvalue $1$,
 \item[(ii)] $\Ft \At_i = \At_i \Ft$ for $1 \leq i \leq g$.
 \item[(iii)] $\Ft \Bt_j = \Bt_j \Ft$ for $2 \leq j \leq g$.
 \item[(iv)] $\Ft \Bt_1 \Ft = \Bt_1 \Ft \Bt_1$.
\end{enumerate}
 \par
 Then, 
 $\Ft = \left(\begin{array}{cc}
	     A_1 & \vect{w} \\ 
	      \transpose{\vect{s}} & 1
	      \end{array}\right)$ 
where $\vect{w}$, $\vect{s} \in \C^{2g}$ with either $\vect{w} = \vect{0}$ or $\vect{s} = \vect{0}$.
\end{theorem}
\par

Conversely, it is clear that $\Ft$ given as the consequence of the theorem satisfies all the conditions (i)-(iv)
in the theorem.

\begin{remark}
 The above conditions (i)-(iv) for $\Ft$ in Theorem \ref{for_extra_gen} 
 resemble but do not coincide with the conditions for $\Xt_1$ in Theorem \ref{prelim_thm:2g+1}.
 In fact, due to the difference, the consequence of the former theorem does not 
 need to take conjugation of $\Ft$ unlike the latter theorem.
\end{remark}

The proofs of Theorems \ref{prelim_thm:2g+1} and \ref{for_extra_gen} are straightforward 
matrix computation, which are elementary but rather long. Therefore, we postpone them to 
Section \ref{braid_commuting}.

 \section{A dichotomy of representations} \label{dichotomy}
\par
In this section, we combine the results in previous sections to complete the proof of 
Theorem \ref{surjectivity-thm}. To do this, the following dichotomy result is crucial.

\begin{theorem} \label{dichotomy_lemma}
 Assume $g \geq 7$. Let $\phi: \M{S} \to \GL{2g+1, \C}$ be any nontrivial linear representation.
 Then, with respect to some basis of $\C^{2g+1}$, one of the following holds:
\begin{itemize}
 \item[(A)] For each $f \in \M{S}$, $\phi(f)$ has the form
		  $$\left(\begin{array}{cc} F & \vect{w} \\ 0 & 1 \end{array}\right)
	               \qquad \text{($F \in \GL{2g, \C}$, $\vect{w} \in \C^{2g}$)}.$$
 \item[(B)] For each $f \in \M{S}$, $\phi(f)$ has the form
	    $$\left(\begin{array}{cc} F & 0 \\ \transpose{\vect{s}} & 1 \end{array}\right)
	               \qquad \text{($F \in \GL{2g, \C}$, $\vect{s} \in \C^{2g}$)}.$$
\end{itemize}
\end{theorem}
\par

\subsection{Proof of Theorem \ref{dichotomy_lemma}}\par

Assume $g \geq 7$. Let $\phi: \M{S} \to \GL{2g+1, \C}$ be an arbitrary non-trivial homomorphism.
We choose a nonseparating simple closed curve $a$ on $S$, and set 
$L_a := \phi(t_a)$. By Theorem \ref{eigen_dim},  $L_a$ has a unique eigenvalue $1$, 
and it holds $\Dm{\ea{1}} = 2g$.
We choose a nonseparating simple closed curve $b$ which intersects $a$ transversely at a single point.
\par
We first observe $\ea{1} \neq \eb{1}$. Indeed, if $\ea{1} = \eb{1}$, then 
Theorems \ref{sru-theorem-3.1} and \ref{ldl-lemma-4.2} imply
$\ea{1} = \eigensp{1}{x}$ for any nonseparating
simple closed curve $x$ on $S$. Since the Dehn twists along
such $x$s' generate $\M{S}$, $\ea{1}$ is $\M{S}$-invariant via $\phi$, and the action of $\M{S}$ on 
$\ea{1}$ is trivial. We may now change the basis of $\C^{2g+1}$ so that its first $2g$ 
elements form a basis of $\ea{1}$ to obtain 
$\phi(f) = \begin{pmatrix}	I & * \\ 0 & 1 	   \end{pmatrix}$ for each $f \in \M{S}$.
This shows $\im{\phi}$ is an abelian group. We then see $\im{\phi}$ is trivial by Theorem 
\ref{abelianization}. This contradicts to $\Dm{\ea{1}} = 2g$, and therefore, we have $\ea{1} \neq \eb{1}$.
\ppar

 Now, to complete the proof of Theorem \ref{dichotomy_lemma}, we have only to 
 prove the theorem for $f$ in the fixed generating set of $\M{S}$ given by Theorem \ref{generators}:
 the Dehn twists along the nonseparating simple closed curves 
 \begin{align*}
  & a_1, b_1, \ldots, a_g, b_g; c_1, c_2, \ldots, c_{g-1}; \\
  & e_1, e_2, \ldots, e_p; f_1, f_2, \ldots, f_r
 \end{align*}
 depicted in Figures \ref{fig_gen_1} and \ref{fig_gen_2}. As usual, we denote $L_c = \phi(t_c)$ 
 for a simple closed curve $c$ on $S$.
\par

 Since $\dim{(E_1^a)} = 2g$, the Jordan form of $L_a$ is given by 
 $$\left(\begin{array}{c|c}
   \begin{matrix}
    1 & 1 \\ 0 & 1
   \end{matrix} & 0  \\
	 \hline 0 & I_{2g-1}
	 \end{array}\right).$$
 Therefore, we can apply Theorem \ref{sru-lemma-4.7} to obtain
 \begin{equation*}
	   L_{a_i} = \left( 
		      \begin{array}{cc}
		       A_i &  0 \\
			0 & 1
		      \end{array}
		     \right), \quad
	   L_{b_i} = \left( 
		      \begin{array}{cc}
	               B_i & 0 \\
		        0  & 1 
		      \end{array}
		     \right) \quad \text{ for $i=1$, $2$, \ldots, $g$}
 \end{equation*}
 after changing the basis of $\C^{2g+1}$ appropriately.
 Here, the matrices $A_i$ and $B_i$ are the ones given in Definition \ref{A_and_B}.
 Then by putting $\Xt_k = L_{c_k}$ for each $k=1$, $2$, \ldots, $g-1$,
 we can apply Theorem \ref{prelim_thm:2g+1} to obtain 
 $$L_{c_k} = \begin{pmatrix} C_k & \vect{w_k} \\ \transpose{\vect{s_k}} & 1 \end{pmatrix}\, 
   \text{ for each $k$}$$
 after changing the basis of $\C^{2g+1}$ further, with $L_{a_i}$ and $L_{b_i}$ unchanged. 
 Here, the matrix $C_k$ is the one given in Definition \ref{L_and_Ck}, and $\vect{w_k}$, 
 $\vect{s_k} \in \C^{2g}$ with either $\vect{w_k} = \vect{0}$ or $\vect{s_k} = \vect{0}$.
 \ppar

 Now we note $c_i$ and $c_j$ are disjoint for any $i \neq j$, and therefore $t_{c_i}$  and $t_{c_j}$, 
 and hence $L_{c_i}$ and $L_{c_j}$, are commutative. By the same reason, the matrices $C_i$ and $C_j$
 are also commutative since they coincide with the images of the Dehn twists $t_{c_i}$ and $t_{c_j}$, 
 respectively, under $\rho_0$ with respect to the basis  $\{x_1, y_1, \ldots, x_g, y_g\}$ given in 
 Definition \ref{symp-basis}. This implies either 
 \begin{equation*}
  \vect{w_1} = \vect{w_2} = \cdots = \vect{w_{g-1}} = \vect{0}
 \end{equation*}
 or 
\begin{equation*}
 \vect{s_1} = \vect{s_2} = \cdots = \vect{s_{g-1}} = \vect{0}.
\end{equation*}
Indeed, for $X$, $Y \in \GL{2g, \C}$ and column vectors $\vect{w}$, $\vect{s} \in \C^{2g}$, 
 we see
\begin{align*}
 \left(\begin{array}{cc} X & \vect{w} \\ 0 & 1 \end{array}\right) 
 \left(\begin{array}{cc} Y &  0 \\  \transpose{\vect{s}} & 1 \end{array}\right) 
 & = \left(\begin{array}{cc} XY + \vect{w} \transpose{\vect{s}} & \vect{w} \\ \transpose{\vect{s}} & 1 \end{array}\right), \\[3pt]
 \left(\begin{array}{cc} Y &  0 \\ \transpose{\vect{s}} & 1 \end{array}\right) 
 \left(\begin{array}{cc} X & \vect{w} \\ 0 & 1 \end{array}\right) 
 & = \left(\begin{array}{cc} YX & Y \vect{w} \\ 
	   \transpose{\vect{s}}X & \transpose{\vect{s}} \vect{w} + 1 \end{array}\right).
\end{align*}
Hence, if $X$ and $Y$ are commutative, comparing the upper left blocks of these two matrix 
multiplications, we see the two matrices $\left(\begin{array}{cc} X & \vect{w} \\ 0 & 1 \end{array}\right)$
and $\left(\begin{array}{cc} Y & 0  \\ \transpose{\vect{s}} & 1 \end{array}\right)$ 
are commutative only if the $2g \times 2g$ matrix $\vect{w} \transpose{\vect{s}}$ 
is equal to the zero matrix, and in that case, it holds either $\vect{w} = \vect{0}$ 
or $\vect{s} = \vect{0}$.
\par
 
 Therefore, we see that either all of $L_{c_i}$s' are of type (A), or all of $L_{c_i}$s' are of type (B).
 
 Next, choose any simple closed curve $e \in \{ e_1, e_2, \ldots, e_p, f_1, f_2, \ldots, f_r \}$,
 and set $L_e := \phi(t_e)$. Then $e$ intersects $b_1$ transversely at a single point and is disjoint from 
 $a_1$, $a_2$, \ldots, $a_g$ and $b_2$, $b_3$, \ldots, $b_g$.
 Hence $t_e$ and $t_{b_1}$ satisfy the braid relation, and $t_e$ commutes with $t_{a_1}$,
 $t_{a_2}$, \ldots, $t_{a_g}$, and $t_{b_2}$, $t_{b_3}$, \ldots, $t_{b_g}$. Therefore, we can apply Theorem
 \ref{for_extra_gen} with $\Ft = L_e$ to obtain either
		  (A) $L_e = \left(\begin{array}{cc} A_1 & \vect{w} \\ 0 & 1 \end{array}\right)$
                          with some $\vect{w} \in \C^{2g}$, or
		  (B) $L_e = \left(\begin{array}{cc} A_1 & 0 \\ \transpose{\vect{s}} & 1 \end{array}\right)$
                          for some $\vect{s} \in \C^{2g}$.
 Furthermore, since $e_1$, $e_2$, \ldots, $e_p$, and $f_1$, $f_2$, \ldots, $f_r$ are pairwise disjoint, the images
 of Dehn twists along them are all commutative. Therefore, by the previous argument above, the types (A) or (B) for
 the images of these Dehn twists are all the same and do not depend on the choice of $e$. 
 \par

 Finally, since $e$ is disjoint from any $c_i$, $L_e$ commutes with $L_{c_i}$. Therefore, the type (A) or (B) for
 $L_e$ coincides with that for $L_{c_i}$s'. This completes the proof of Theorem \ref{dichotomy_lemma}.
\qed
\ppar

\subsection{Proof of Theorem \ref{surjectivity-thm}}\par

We now prove Theorem \ref{surjectivity-thm}. Let $\phi: \M{S} \to \GL{2g+1, \C}$ be any nontrivial linear
representation. We need to prove that either $[\phi] = \sigma([\phi_c])$ or 
$\iota([\phi]) = \sigma([\phi_c])$ for some crossed homomorphism $c: \M{S} \to H_\C = \C^{2g}$. By Theorem 
\ref{dichotomy_lemma}, we have only to consider the cases (A) and (B) in the theorem.

\subsubsection*{The case (A)} \par
For each $f \in \M{S}$, 
$\phi(f)$ has the form 
	 $$\phi(f) = \left(\begin{array}{cc} F &  \vect{w} \\ 0 & 1 \end{array}\right)
	 \text{with $F \in \GL{2g, \C}$ and $\vect{w} \in \C^{2g}$.}$$
The correspondence $f \mapsto F$ defines a linear representation
$\bar{\phi} : \M{S} \to \GL{2g,\C}$. By Theorem \ref{2g-rep}, 
$\bar{\phi}$ is trivial or conjugate to the symplectic representation $\rho_0$.

If $\bar{\phi}$ is trivial, then $\im{\phi}$ is abelian, and since $g \geq 3$, $\phi$ is trivial by
Theorem \ref{abelianization}, which contradicts to the assumption. Therefore,
we see $\bar{\phi}$ is conjugate to $\rho_0$.

By changing the basis of $\C^{2g+1}$ if necessary, we may assume $\bar{\phi}$ coincides with the matrix form
of $\rho_0$ with respect to the basis $\{ x_1, y_1, \ldots, x_g, y_g \}$ of $H_\C$.
Then we may consider $\vect{w} \in H_\C$, and the correspondence  $f \in \M{S} \mapsto \vect{w}$ defines
a crossed homomorphism 
$$c: \M{S} \to H_\C$$
with values in $H_\C$.
Namely, it holds 
$$c(f_1 f_2) = c(f_1) + \symp(f_1) c(f_2) \qquad \text{($f_1$, $f_2 \in \M{S}$).}$$
Therefore, we have $[\phi] = [\phi_c] = \sigma([c])$ in $X_0$.
\ppar

\subsubsection*{The case (B)}\par

For each $f \in \M{S}$,  $\phi(f)$ has the form 
$$\phi(f) = \left(\begin{array}{cc} F & 0 \\ \transpose{\vect{s}} & 1 \end{array}\right) \quad
     \text{with $F \in \GL{2g, \C}$ and $\vect{s} \in \C^{2g}$}.$$
\par

Let $\phi^* : \M{S} \to \GL{2g+1, \C}$ denote the dual representation of $\phi$ defined by 
$$ \phi^*(f) = \transpose{\phi(f)^{-1}} \quad \text{for each $f \in \M{S}$.}$$
Then $\phi^*$ is clearly a representation of Type (A), and hence we can apply the previous argument 
to obtain $[\phi^*] = \sigma([c])$ for some crossed homomorphism $c$. In other words, we have 
$\iota([\phi]) = \sigma([c])$ in $X_0$. This completes the proof of Theorem \ref{surjectivity-thm}.
\qed

 \section{Braid and commuting relations in matrices} \label{braid_commuting} \par

In this section we prove Theorems \ref{prelim_thm:2g+1} and \ref{for_extra_gen}. We first recall 
necessary results of Korkmaz, which was originally used for proving Theorem \ref{2g-rep}.

\subsection{Preliminary from \cite{sru}}\par

The next theorem follows from the irreducibility of the symplectic representation $\rho_0$ for $g=1$ 
together with Schur's lemma, or alternatively, can be verified by straightforward computation.
\par

\begin{theorem}[\cite{sru}, Lemma 2.2]\label{sru-lemma-2.2}
 Let $X$, $Y$ and $Z$ be $2 \times k$, $k \times 2$ and $2 \times 2$ matrices
 with entries in $\C$, respectively.
 \begin{itemize}
  \item[(1)] If \/ $UX = X$ and $\Uh X = X$, then $X=0$. 
  \item[(2)] If \/ $YU = Y$ and $Y \Uh = Y$, then $Y=0$.
  \item[(3)] If \/ $ZU = UZ$ and $Z \Uh = \Uh Z$, then $Z = \alpha I_2$ for some $\alpha \in \C$.
 \end{itemize}
\end{theorem}
This theorem can be generalized as follows by induction on $g$.

\begin{theorem}[\cite{sru},Lemma 2.3]\label{sru-lemma-2.3}
 Let $X$, $Y$, $Z$ be matrices with entries in $\C$ such that the multiplications given below 
 are all defined.
 \begin{itemize}
  \item[(1)] If \/ $A_i X = X$ and $B_i X = X$ \/ for $1 \leq i \leq g$, then \/ $X=0$.
  \item[(2)] If \/ $Y A_i = Y$ and $Y B_i = Y$ \/ for $1 \leq i \leq g$, then \/ $Y = 0$.
  \item[(3)] If \/ $Z A_i = A_i Z$ and $Z B_i = B_i Z$ \/ for $1 \leq i \leq g$, \/
	then \/ $Z = \diag{\alpha_1 I_2, \alpha_2 I_2, \ldots, \alpha_g I_2}$ \/ for some
	$\alpha_1$, $\alpha_2$, \ldots, $\alpha_g \in \C$.
 \end{itemize}
\end{theorem}
\par

We remark that Theorem \ref{sru-lemma-2.3} does not assume any of $X$, $Y$ and $Z$ represents a
$\M{S}$-homomorphism, unlike Theorem \ref{sru-lemma-2.2}.
In fact, the assumption of Theorem \ref{sru-lemma-2.2} implies, for instance, $Z$ represents a 
$\M{S_{1,0}^0}$-endomorphism of $H_\C$ because $U$ and $\Uh$ are, respectively, 
the images of the Dehn twists $t_{a_1}$ and $t_{b_1}$, which generate $\M{S_{1,0}^0}$, 
under the symplectic representation with respect to the basis $\{x_1, y_1\}$ given in Definition
\ref{symp-basis}. Therefore the consequence of the theorem follows from Schur's lemma. 
On the other hand, the assumption of Theorem \ref{sru-lemma-2.3}, for $g \geq 2$, 
does not imply $Z$ represents 
a $\M{S_{g,0}^0}$-endomorphism of $H_\C$ since otherwise the consequence of the theorem could 
be strengthened to $Z = aI$ rather than a block diagonal matrix, which is not necessarily true. 
The difference between the two theorems is due to the fact that the Dehn twists 
$t_{a_1}$, $t_{a_2}$, \ldots, $t_{a_g}$ and $t_{b_1}$, $t_{b_2}$, \ldots, $t_{b_g}$ are not sufficient 
to generate $\M{S_{g,0}^0}$ if $g \geq 2$.
\ppar

\subsection{A key lemma}\par

A key step to prove Theorem \ref{prelim_thm:2g+1} is the following, which 
is to characterize the matrix satisfying the conditions for $\Xt_1$ in Theorem \ref{prelim_thm:2g+1}.
\par
\begin{lemma} \label{key_lemma}
 Let $g \geq 2$, $m \geq 1$, and $\Xt \in \GL{2g+m, \C}$. Let 
 $$\At_i = \begin{pmatrix} A_i & 0 \\ 0 & I_m \end{pmatrix} \quad\text{and}\quad
    \Bt_i = \begin{pmatrix} B_i & 0 \\ 0 & I_m \end{pmatrix} \quad\text{for $1 \leq i \leq g$.}$$
 Suppose $\Xt$  satisfies the following conditions (i)-(iv).
\begin{enumerate}
   \item[(i)] $\Xt$ has a unique eigenvalue $1$.
   \item[(ii)] $\Xt \At_i = \At_i \Xt$ \/ for \/ $1 \leq i \leq g$.
   \item[(iii)] If $g \geq 3$, then $\Xt \Bt_j = \Bt_j \Xt$ \/ for \/ $3 \leq j \leq g$.
   \item[(iv)] $\Xt \Bt_j \Xt = \Bt_j \Xt \Bt_j$ \/ for $j=1$, $2$.
  \end{enumerate}
Then, there exists a nonzero complex number $p$ such that for 
 $$P = \diag{I_2, pI_2, I_{2g-4}} \quad\text{and } 
 \quad \Pt = \begin{pmatrix} P & 0 \\ 0 & I_m \end{pmatrix},$$ 
 it holds
 \begin{equation*}
  \Pt^{-1}\Xt\Pt = \left(\begin{array}{cc} C_1 & W_1 \\ S_1 & T \end{array}\right)
 \end{equation*}
 where $C_1$ and $T$ are, respectively, $2g \times 2g$ and $m \times m$ matrices, and 
 $$
 C_1  = \begin{pmatrix} L & 0 \\ 0 & I_{2g-4} \end{pmatrix} \, 
       \text{with $L$ the $4 \times 4$ matrix given in Definition \ref{L_and_Ck},}$$
 $$
 W_1 = \transpose{\begin{pmatrix} \vect{w} & \vect{0} & -\vect{w} & \vect{0} & 
		  \vect{0} & \cdots & \vect{0} \end{pmatrix}}, \quad
  S_1 = \begin{pmatrix} \vect{0} & \vect{s} & \vect{0} & -\vect{s} & 
	\vect{0} & \cdots & \vect{0} \end{pmatrix}$$
 with 
 $\vect{w}$, $\vect{s}$, $\vect{0} \in \C^m$; $\transpose{\vect{w}} \vect{s} = 0$; 
 $\transpose{\vect{w}}T = \transpose{\vect{w}}$, $T \vect{s} = \vect{s}$; 
and $T^2 - T = \vect{s} \transpose{\vect{w}}$.
\end{lemma}
\ppar

To prove the lemma, we first observe:
\begin{lemma} \label{Lem_commute_with_U}
 Let $Y = \begin{pmatrix}a & b\\ c & d \end{pmatrix}$ with $a$, $b$, $d$, $d \in \C$.
 \begin{enumerate}
  \item[(1)] If $YU = UY$, then $a=d$ and $c = 0$ so that $Y = \begin{pmatrix}a & b \\ 0 & a \end{pmatrix}$.
  \item[(2)] If $YU = UY = Y$, then $a=d=c=0$ so that $Y = \begin{pmatrix} 0 & b \\ 0 & 0 \end{pmatrix}$.
 \end{enumerate}
\end{lemma}

\begin{proof}
 Straightforward.
\end{proof}
\ppar

\subsubsection*{Proof of Lemma \ref{key_lemma}}
\par

We write $\Xt = \begin{pmatrix} X & W \\ S & T \end{pmatrix}$
where $X$ and $T$ are, respectively, $2g \times 2g$ and $m \times m$ matrices.
By a straightforward computation, we see the condition (ii) implies
  \begin{align}
   & X A_i  = A_i X,  \label{c_1} \\
   & W = A_i W,  \label{c_2} \\
   & S A_i  = S \label{c_3} 
  \end{align}
for $1 \leq i \leq g$.
\par

In case $g \geq 3$, the condition (iii) similarly implies
  \begin{align}
   & X B_j  = B_j X, \label{c_4} \\
   & B_j W  = W,  \label{c_5} \\
   & S B_j  = S \label{c_6} 
  \end{align}
for $3 \leq j \leq g$.
\par

The condition (iv) implies
  \begin{align}
   & X B_j X + WS = B_j X B_j, \label{c_7} \\
   & X B_j W + WT = B_j W,  \label{c_8} \\
   & S B_j X + TS = S B_j, \label{c_9} \\
   & S B_j W + T^2 = T \label{c_10}
  \end{align}
for $j=1$, $2$.
\ppar

In view of \eqref{c_3} and \eqref{c_6}, we have 
\begin{itemize}
 \item $S(A_i - I_{2g}) = 0$ for $1 \leq i \leq g$,
 \item $S(B_j - I_{2g}) = 0$ for $3 \leq j \leq g$ if $g \geq 3$.
\end{itemize}
Therefore, we can easily see that all entries of $S$ are zero except in the second and the fourth columns.
Thus, we may write
\begin{equation} \label{form_S}
     S = \begin{pmatrix} \vect{0} & \vect{s_2} & \vect{0} & \vect{s_4} & \vect{0} 
	 & \cdots & \vect{0} \end{pmatrix} \quad 
        \text{where $\vect{s_2}$, $\vect{s_4}$, $\vect{0} \in \C^m$.}
\end{equation}
Similarly, in view of \eqref{c_2} and \eqref{c_5}, we can see all the entries of $W$ are zero except in the 
first and the third rows. We may thus write
\begin{equation} \label{form_W}
 W = \transpose{\begin{pmatrix}	\vect{w_1} & \vect{0} & \vect{w_3} & \vect{0} & 
		\vect{0} & \cdots & \vect{0} \end{pmatrix}} \quad
       \text{ where $\vect{w_1}$, $\vect{w_3}$, $\vect{0} \in \C^m$.}
\end{equation}

We then have
\begin{align}
 WS & = \left(\begin{array}{c|c}
     \begin{matrix}
      0 & \transpose{\vect{w_1}} \vect{s_2} & 0 & \transpose{\vect{w_1}} \vect{s_4} \\
      0 &  0 & 0  &  0 \\
      0 & \transpose{\vect{w_3}} \vect{s_2} & 0 & \transpose{\vect{w_3}} \vect{s_4} \\
      0 &  0 & 0  &  0 
     \end{matrix} & 0  \\
	 \hline 0 & 0
	\end{array}\right) \label{WS_formula}
\end{align}
where the upper left block is a $4 \times 4$ matrix.
\ppar

Next, we consider the form of $X$. Suppose $g \geq 3$ for the moment,  and we write 
$$X = \begin{pmatrix} X_0 & W_0\\ S_0 & T_0 \end{pmatrix}$$
where $X_0$ and $T_0$ are, respectively, $4 \times 4$ and $(2g-4) \times (2g-4)$ matrices. 
For $3 \leq i \leq g$, we set $(2g-4) \times (2g-4)$ matrices $\Ab_i$ and $\Bb_i$ by
$$ A_i =\begin{pmatrix} I_4 & 0 \\ 0 & \Ab_i \end{pmatrix}, \qquad 
	     B_i = \begin{pmatrix} I_4 & 0 \\ 0 & \Bb_i \end{pmatrix}.$$
Then the equalities \eqref{c_1} for $i \geq 3$ and \eqref{c_4} imply
\begin{itemize}
 \item $W_0 \Ab_i = W_0 = W_0 \Bb_i$,
 \item $S_0 = \Ab_i S_0 = \Bb_i S_0$,
 \item $T_0 \Ab_i = \Ab_i T_0$ and $T_0 \Bb_i = \Bb_i T_0$
\end{itemize}
for $3 \leq i \leq g$. We can then apply Theorem \ref{sru-lemma-2.3} to obtain
 $W_0 = 0$, $S_0 = 0$, and $T_0 = \diag{\alpha_1 I_2, \alpha_2 I_2, \ldots, \alpha_{g-2} I_2}$
for some $\alpha_1$, $\alpha_2$, \ldots, $\alpha_{g-2} \in \C$.
In view of the form of $S$ \eqref{form_S}, we see each $\alpha_i$ is an eigenvalue of $\Xt$, and 
therefore, $\alpha_i = 1$ by the condition (i). In case $g = 2$, we may simply set $X_0 = X$. In short, 
we conclude for $g \geq 2$ that
\begin{equation}
 X = \begin{pmatrix} X_0 & 0 \\ 0 & I_{2g-4} \end{pmatrix} \label{form_X}
\end{equation}
where $X_0$ is a $4 \times 4$ matrix.
\par

We next write $X_0 = \begin{pmatrix} X_1 & X_2 \\ X_3 & X_4 \end{pmatrix}$ where $X_1$ and 
$X_4$ are $2 \times 2$ matrices. Then the equality 
\eqref{c_1} for $i=1$, $2$ implies
\begin{itemize}
 \item $X_1 U = U X_1$,
 \item $X_2 U = X_2 = U X_2$,
 \item $X_3 U = X_3 = U X_3$,
 \item $X_4 U = U X_4$.
\end{itemize}
We can therefore apply Lemma \ref{Lem_commute_with_U} to obtain
\begin{equation*} \label{form_X_0}
 X_0 = \begin{pmatrix}
	a & b & 0  & \alpha \\
	0 & a & 0 & 0 \\
	0 & \beta & c  & d \\
	0 & 0 & 0  & c
       \end{pmatrix}
\quad \text{for some $\alpha$, $\beta$, $a$, $b$, $c$, $d \in \C$.}
\end{equation*}
In particular, in view of the first and the third columns of $X_0$ together with \eqref{form_X}
and \eqref{form_S}, we see $a$ and $c$ are both eigenvalues of $\Xt$, and therefore, we have 
$a = c = 1$ by the condition (i). We may thus write 
\begin{equation} \label{form_X0}
	X_0 = \begin{pmatrix}
	       I_2 + b (U - I_2) & \alpha (U - I_2) \\
	       \beta (U - I_2) & I_2 + d (U - I_2)
	      \end{pmatrix}.
\end{equation}
\par

Next, we consider the equality \eqref{c_7}. We set the $4 \times 4$ matrix $B_j^\prime$ as
$$B_j = \begin{pmatrix} B_j^\prime & 0 \\ 0 & I_{2g-4} \end{pmatrix} \quad\text{for $j=1$, $2$.}$$
In view of \eqref{form_X} and \eqref{WS_formula}, the equality \eqref{c_7} implies
\begin{equation} \label{c_7_modified}
 X_0 B_j^\prime X_0 + (WS)_{1,1} = B_j^\prime X_0 B_j^\prime \quad\text{for $j=1$, $2$}
\end{equation}
where $(WS)_{1,1}$ denotes the upper left block of $WS$ in \eqref{WS_formula}, and can be written as
\begin{equation*}
 (WS)_{1,1} = \begin{pmatrix}
	       \transpose{\vect{w_1}} \vect{s_2} (U - I_2) & \transpose{\vect{w_1}} \vect{s_4} (U - I_2) \\
	       \transpose{\vect{w_3}} \vect{s_2} (U - I_2) & \transpose{\vect{w_3}} \vect{s_4} (U - I_2)
	      \end{pmatrix}.
\end{equation*}
Then, together with \eqref{form_X0} and obvious equalities $( U - I_2 )^2 = (\Uh -I_2)^2 = 0$,
$(U-I_2) \Uh (U-I_2) = (-1)(U-I_2)$, and $(\Uh - I_2) U (\Uh - I_2) = (-1)( \Uh - I_2)$, \eqref{c_7_modified} 
for $j=1$ implies the following:
\begin{align}
   \begin{split} \label{c_11}
    & \Uh  +  b \Uh (U-I_2) + b (U-I_2) \Uh - b^2 (U-I_2) + \transpose{\vect{w_1}} \vect{s_2} (U-I_2) \\
     & \qquad  = \Uh^2 + b \Uh (U-I_2) \Uh,
   \end{split} \\
 &  \alpha \Uh (U-I_2)+ (1-b) \alpha (U-I_2) + \transpose{\vect{w_1}} \vect{s_4} (U-I_2)
    = \alpha \Uh (U-I_2), \label{c_12} \\
 & \beta (U-I_2) \Uh + \beta (1-b) (U-I_2) + \transpose{\vect{w_3}} \vect{s_2} (U-I_2) 
    = \beta (U-I_2) \Uh, \label{c_13} \\
 &  I_2 + (2d - \alpha \beta) (U-I_2) + \transpose{\vect{w_3}}\vect{s_4} (U-I_2) 
    = I_2 + d (U-I_2). \label{c_14}
\end{align}
\par

Similarly, the lower right block of \eqref{c_7_modified} for $j=2$ implies
\begin{align}
 \begin{split} \label{c_15}
   &  \Uh + d \Uh (U-I_2) + d (U-I_2) \Uh - d^2 (U-I_2) + \transpose{\vect{w_3}} \vect{s_4} (U-I_2) \\
   & \qquad = \Uh^2 + d \Uh (U-I_2) \Uh.
 \end{split}
\end{align}
\par

By straightforward computations of the $(2,1)$-entries of \eqref{c_11} and \eqref{c_15}, 
we obtain $b=1$ and $d=1$, respectively. Then, straightforward computations of the $(1,2)$-entries 
of \eqref{c_11}, \eqref{c_12}, \eqref{c_13}, and \eqref{c_15} imply in turn 
$\transpose{\vect{w_1}} \vect{s_2}$, $\transpose{\vect{w_1}} \vect{s_4}$, 
$\transpose{\vect{w_3}} \vect{s_2}$, and $\transpose{\vect{w_3}} \vect{s_4}$ are all zero.
This means simply 
\begin{equation} \label{c_16}
 WS = 0.
\end{equation}
Furthermore, in view of the $(1,2)$-entry of \eqref{c_14}, we obtain $\alpha \beta = 1$.
Therefore, we have 
$$  
   X= \left(\begin{array}{c|c}
    \begin{matrix}
     1 & 1 & 0 & \alpha \\
     0 & 1 & 0 & 0 \\
     0 & \frac{1}{\alpha} & 1 & 1 \\
     0 & 0 & 0 & 1
    \end{matrix} & 0 \\
\hline 0 & I_{2g-4}
      \end{array}\right) \quad\text{with $\alpha \neq 0$}$$
and 
$\Xt = \begin{pmatrix} X & W \\ S & T \end{pmatrix}$
where $W$ and $S$ are as \eqref{form_W} and  \eqref{form_S} with $WS = 0$.
\ppar

Now, we can see by tedious but straightforward computation that the equalities \eqref{c_8},
\eqref{c_9}, and \eqref{c_10} are, respectively, equivalent to 
\begin{align}
 & \vect{w_3}  = \frac{1}{\alpha} \vect{w_1} \quad\text{and}\quad \transpose{\vect{w_1}} T 
                  = \transpose{\vect{w_1}},\label{c_8_prime} \\
 & \vect{s_4}  = \alpha \vect{s_2} \quad\text{and}\quad T \vect{s_2} = \vect{s_2}, \label{c_9_prime} \\
 \intertext{and}
 & \vect{s_2} \transpose{\vect{w_1}} = T^2 - T = \vect{s_4} \transpose{\vect{w_3}}. \label{c_10_prime}
\end{align}
\ppar

Finally, let $p = -\frac{1}{\alpha}$, $P = \diag{I_2, p I_2, I_{2g-4}}$, and 
$\Pt = \begin{pmatrix} P & 0 \\ 0 & I_m \end{pmatrix}$. Also, let $\vect{s} = \vect{s_2}$ and 
$\vect{w} = \vect{w_1}$. Then a direct computation implies
$$\Pt^{-1} \Xt \Pt = \begin{pmatrix} C_1 & W_1 \\ S_1 & T \end{pmatrix},$$
where
\begin{align*}
 W_1 & 
      = P^{-1}W = \transpose{\begin{pmatrix} \vect{w} & \vect{0} & -\vect{w} & \vect{0} & 
			     \vect{0} & \cdots & \vect{0} \end{pmatrix}}, \\
 S_1 & = SP = \begin{pmatrix} \vect{0} & \vect{s} & \vect{0} & -\vect{s} & 
	      \vect{0} & \cdots & \vect{0} \end{pmatrix}.
\end{align*}
Furthermore, by \eqref{c_16} and \eqref{c_8_prime}-\eqref{c_10_prime}, we have
$$\transpose{\vect{w}} \vect{s} = 0, \quad \transpose{\vect{w}} T = \transpose{\vect{w}}, 
    T \vect{s} = \vect{s}, \quad\text{and}\quad T^2 - T = \vect{s} \vect{w}.$$
This completes the proof of Lemma \ref{key_lemma}. \qed
\ppar

\begin{remark} \label{weakness}
 In general, the conclusion of Lemma \ref{key_lemma} cannot be strengthened to 
 ``either $W_1 = 0$ or $S_1 = 0$.'' In fact, a counterexample is given for $m=2$ by 
  $$\Xt = \begin{pmatrix} C_1 & W_1 \\ S_1 & T\end{pmatrix} \, \text{ with } \,
	T = \begin{pmatrix} 2 & 1 \\ -1 & 0\end{pmatrix},$$ 
 and
 \begin{equation*}
  W_1  = \transpose{\begin{pmatrix} \vect{w} & \vect{0} & -\vect{w} & \vect{0} & 
		    \vect{0} & \cdots & \vect{0} \end{pmatrix}}, \quad
  S_1  = \begin{pmatrix} \vect{0} & \vect{s} & \vect{0} & -\vect{s} & 
	 \vect{0} & \cdots & \vect{0} \end{pmatrix}
 \end{equation*}
 where $\vect{w} = \begin{pmatrix} 1 \\ 1 \end{pmatrix}$ and 
	     $\vect{s} = \begin{pmatrix} 1 \\ -1 \end{pmatrix}$.
\end{remark}

\subsection{Proof of Theorem \ref{prelim_thm:2g+1}} \par

We first observe that the conclusion of Lemma \ref{key_lemma} can be strengthened if $m=1$.
Here, we denote
\begin{equation*}
 \At_i = \begin{pmatrix} A_i & 0 \\ 0 & 1 \end{pmatrix} \quad\text{and}\quad
 \Bt_i = \begin{pmatrix} B_i & 0 \\ 0 & 1 \end{pmatrix} \in \GL{2g+1, \C}
\end{equation*}
for $i=1$, $2$, \ldots, $g$.

\begin{theorem} \label{consequence:2g+1_1}
 Let $g \geq 2$. Suppose $\Xt \in \GL{2g+1, \C}$ satisfies the following conditions (i)--(iv).
\begin{enumerate}
   \item[(i)] $\Xt$ has a unique eigenvalue $1$.
   \item[(ii)] $\Xt \At_i = \At_i \Xt$ for $1 \leq i \leq g$.
   \item[(iii)] If $g \geq 3$, then $\Xt \Bt_j = \Bt_j \Xt$ for $3 \leq j \leq g$.
   \item[(iv)] $\Xt \Bt_j \Xt = \Bt_j \Xt \Bt_j$ for $j=1$, $2$.
  \end{enumerate}
 Then there exists a nonzero complex number $p$ such that for $P=\diag{I_2, p I_2, I_{2g-4}}$ 
 and $\Pt = \begin{pmatrix} P & 0 \\ 0 & 1 \end{pmatrix}$, it holds
 \begin{equation*}
  \Pt^{-1} \Xt \Pt = \begin{pmatrix} C_1 & \vect{w} \\ \transpose{\vect{s}} & 1 \end{pmatrix}
 \end{equation*}
 where $\vect{w}$, $\vect{s} \in \C^{2g}$ with either $\vect{w} = \vect{0}$ or $\vect{s} = \vect{0}$,
 and the entries of $\vect{w}$ and $\vect{s}$ are all zero except for the first through fourth rows.
\end{theorem}

\begin{proof}
 By Lemma \ref{key_lemma} with $m=1$, there exists a nonzero complex number $p$ such that for $P$ and 
 $\Pt$ as  in the theorem, it holds
 \begin{equation*}
  \Pt^{-1} \Xt \Pt = \begin{pmatrix} C_1 & \vect{w} \\ \transpose{\vect{s}} & t \end{pmatrix}
 \end{equation*}
 where $\vect{w}$, $\vect{s} \in \C^{2g}$ and $t \in \C$ with the properties given in the lemma, 
 among which
 \begin{align*}
  \transpose{\vect{w}} & = \begin{pmatrix} w & 0 & -w &  0 & 0 & \cdots & 0 \end{pmatrix}, \\
  \transpose{\vect{s}} & = \begin{pmatrix} 0 & s &  0 & -s & 0 & \cdots & 0 \end{pmatrix}
 \end{align*}
 for some $w$, $s \in \C$ with $ws = 0$. This implies either $\vect{w} = \vect{0}$ or $\vect{s} = \vect{0}$.
 We then have $t=1$ by considering the determinant of $\Pt^{-1} \Xt \Pt$. This completes the proof.
\end{proof}
\ppar

 We note in case $g=2$, Theorem \ref{consequence:2g+1_1} implies Theorem \ref{prelim_thm:2g+1} 
 by putting $\Xt = \Xt_1$. 
 Hereafter in this subsection, we assume $g \geq 3$, and the index $i$ is considered modulo $g$.
We consider a certain periodic homeomorphism of the closed surface of $\Sb$. Recall that 
$\Sb$ is obtained from the surface $S$ by gluing a $2$-disk to each boundary component and 
forgetting all the punctures. We configure $\Sb$ as in Figure \ref{fig3} as well as simple closed curves 
$a_i$, $b_i$, $c_k$ for $1 \leq i \leq g$ and $1 \leq k \leq g-1$. We denote by $r \in \M{\Sb}$
the mapping class represented by the counter-clockwise $1/g$-rotation around the center. 
 \begin{figure}
  \centering
  \includegraphics[scale=1.0]{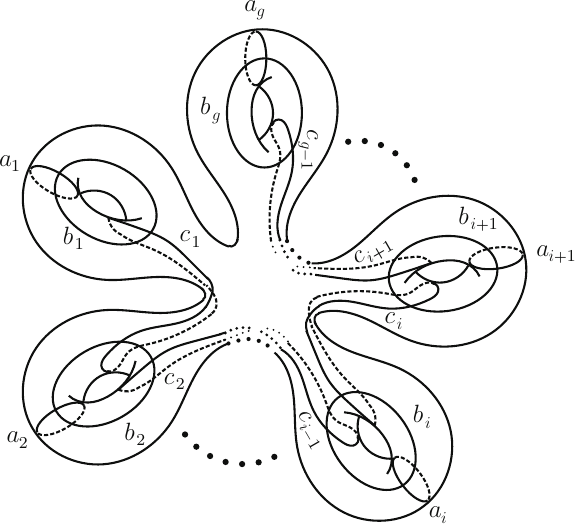}
  \caption{Periodic homeomorphism $r$}\label{fig3}
 \end{figure}
We see
 $$r(a_i) = a_{i+1}, \, r(b_i) = b_{i+1}, \, \text{ and } \, r(c_k) = c_{k+1}$$
for $1 \leq i \leq g$ and $1 \leq k \leq g-2$.
We now denote the symplectic representation of $\M{\Sb}$ by $\bar{\rho}_0: \M{\Sb} \to \GL{2g, \C}$.
Let $G = \bar{\rho}_0(r)$ with respect to the basis $\{x_i, y_i\}$ as in Definition \ref{symp-basis} with $S$
replaced by $\Sb$ where the curves $x_i$ and $y_i$ are also reconfigured in an obvious manner.
One can easily see $G = \begin{pmatrix} 0 & I_2 \\ I_{2g-2} & 0 \end{pmatrix}$.
Also, let $\Gt = \begin{pmatrix} G & 0 \\ 0 & 1\end{pmatrix}$. We note
\begin{equation*}
 A_i = \bar{\rho}_0(t_{a_i}), \, B_i = \bar{\rho}_0(t_{b_i}), \, \text{ and } \, C_k = \bar{\rho}_0(t_{c_k})
\end{equation*}
for $1 \leq i \leq g$ and for $1 \leq k \leq g-1$
with respect to the same basis. Therefore, by making use of Theorem \ref{dtwist-rel} (1), 
we see for instance 
$$G^{-1} A_i G = \bar{\rho}_0(r^{-1} t_{a_i} r) = \bar{\rho}_0(t_{r^{-1}(a_i)}) 
    = \bar{\rho}_0(t_{a_{i-1}}) = A_{i-1} \quad\text{for $1 \leq i \leq g$.}$$
Combining similar computations, we obtain for $1 \leq i \leq g$ \/ and $2 \leq k \leq g-1$,
\begin{align*}
 &  G^{-1} A_i G = A_{i-1}, \,  G^{-1} B_i G = B_{i-1}, \, \text{ and }\, G^{-1} C_k G = C_{k-1} \\
 \intertext{as well as}
 &  \Gt^{-1} \At_i \Gt = \At_{i-1}, \, \Gt^{-1} \Bt_i \Gt = \Bt_{i-1}. &
\end{align*}
\ppar

We can now begin the proof of Theorem \ref{prelim_thm:2g+1} for $g \geq 3$.
Suppose the matrices $\Xt_1$, $\Xt_2$, \ldots, $\Xt_{g-1} \in \GL{2g+1, \C}$ satisfy the conditions
(i)-(iv) in the theorem. 
\par

We first take $\Xt = \Xt_1$ and apply Theorem \ref{consequence:2g+1_1} to obtain $p_1 \in \C^\times$,
$P_1 = \diag{I_2, p_1 I_2, I_{2g-4}}$, and $\Pt_1 = \begin{pmatrix} P & 0 \\ 0 & 1 \end{pmatrix}$ so that
$$\Pt_1^{-1} \Xt_1 \Pt_1 = \begin{pmatrix} C_1 & \vect{w_1} \\ \transpose{\vect{s_1}} & 1 \end{pmatrix}$$
where $\vect{w_1}$, $\vect{s_1} \in \C^{2g}$ with either $\vect{w_1} = \vect{0}$ or $\vect{s_1} = \vect{0}$, 
and the entries of $\vect{w_1}$ and $\vect{s_1}$ are all zero except for the first through fourth rows.
\par

Suppose next that nonzero complex numbers $p_i$ for $1 \leq i \leq k$ with $1 \leq k < g-1$ are 
provided so that for 
\begin{align*}
 P_i & = \diag{I_{2i}, p_i I_2, I_{2g-2i-2}}, \\
 Q_k & = P_1 P_2 \cdots P_k, \, \text{and } \, \Qt_k = \begin{pmatrix} Q_k & 0  \\ 0 & 1 \end{pmatrix},
\end{align*}
$$\Qt_k^{-1} \Xt_i \Qt_k = \begin{pmatrix} C_k & \vect{w_i} \\ \transpose{\vect{s_i}} & 1 \end{pmatrix}
  \quad\text{for $1 \leq i \leq k$}$$
where $\vect{w_i}$, $\vect{s_i} \in \C^{2g}$ with either $\vect{w_i} = \vect{0}$ or $\vect{s_i} = \vect{0}$, 
and the entries of $\vect{w_i}$ and $\vect{s_i}$ are all zero except for the $(2i-1)$st through 
$(2i+2)$nd rows. Under this assumption, we seek $p_{k+1} \in \C^\times$ which produces appropriate 
$P_{k+1}$, $Q_{k+1}$, and $\Qt_{k+1}$. 

We first observe that all of $A_i$ and $B_i$ are block diagonal matrices with each diagonal block a 
$2 \times 2$ matrix. Therefore, all of $A_i$ and $B_i$ commute with $P_1$, $P_2$, \ldots, $P_k$.
We hence have
$$ \Qt_k^{-1} \At_i \Qt_k = \At_i \, \text{ and } \, \Qt_k^{-1} \Bt_i \Qt_k = \Bt_i 
        \quad\text{for $1 \leq i \leq g$.}$$
Now, let $\Xt_{k+1}^\prime = \Qt_k^{-1} \Xt_{k+1} \Qt_k$. Then by taking the conjugation of the conditions 
(ii)-(iv) by $\Qt_k$, we see that $\Xt_{k+1}^\prime$ satisfies the same conditions (i)-(iv) for $\Xt_{k+1}$.
We set further 
$$ \Xt = \left( \Gt^k \right)^{-1} \Xt_{k+1}^\prime \Gt^k.$$
Since $\left( \Gt^k \right)^{-1} \Bt_{k+1} \Gt^k = \Bt_1$ and 
$\left( \Gt^k \right)^{-1} \Bt_{k+2} \Gt^k = \Bt_2$ in particular, the matrix $\Xt$ satisfies the assumption 
of Theorem \ref{consequence:2g+1_1}. Therefore, we can apply Theorem \ref{consequence:2g+1_1} to 
obtain $p_{k+1} \in \C^\times$ so that for 
 $P  = \diag{I_2, p_{k+1} I_2, I_{2g-4}}$ and $\Pt = \begin{pmatrix} P & 0 \\ 0 & 1 \end{pmatrix}$,
$$\Pt^{-1} \left( \Gt^k \right)^{-1} \Xt_{k+1}^\prime \Gt^k \Pt 
     = \begin{pmatrix} C_1 & \vect{w_{k+1}^\prime} \\ \transpose{\vect{s_{k+1}^\prime}} & 1 \end{pmatrix}$$
where $\vect{w_{k+1}^\prime}$, $\vect{s_{k+1}^\prime} \in \C^{2g}$ with either 
$\vect{w_{k+1}^\prime} = \vect{0}$, or $\vect{s_{k+1}^\prime} = \vect{0}$, and the entries of 
$\vect{w_{k+1}^\prime}$ and $\vect{s_{k+1}^\prime}$ are all zero except for the first through fourth rows.
On the other hand, it is easy to see 
$$ G^k P (G^k)^{-1} = \diag{I_{2k+2}, p_{k+1} I_2, I_{2g-2k-4}},$$
which we denote by $P_{k+1}$. We further set
\begin{equation*}
 \Pt_{k+1}  = \begin{pmatrix} P_{k+1} & 0 \\ 0 & 1 \end{pmatrix}, \, Q_{k+1} = P_1 P_2 \cdots P_{k+1}, \,
\text{and }\, \Qt_{k+1}  = \begin{pmatrix} Q_{k+1} & 0 \\ 0 & 1 \end{pmatrix}.
\end{equation*}
We then compute
\begin{align*}
 \Qt_{k+1}^{-1} \Xt_{k+1} \Qt_{k+1} & = \Pt_{k+1}^{-1} \Xt_{k+1}^\prime \Pt_{k+1} \\
       & = \Gt^k \left( \Pt^{-1} \left(\Gt^k \right)^{-1} \Xt_{k+1}^\prime \Gt^k \Pt \right)
           \left( \Gt^k \right)^{-1} \\
       & = \begin{pmatrix} G^k & 0 \\ 0 & 1\end{pmatrix} 
           \begin{pmatrix} C_1 & \vect{w_{k+1}^\prime} \\ \transpose{\vect{s_{k+1}^\prime}} & 1 \end{pmatrix}
           \begin{pmatrix} ( G^k )^{-1} & 0 \\ 0 & 1\end{pmatrix} \\
       & = \begin{pmatrix} G^k C_1 (G^k)^{-1} & G^k \vect{w_{k+1}^\prime} \\ 
               \transpose{\vect{s_{k+1}^\prime}} ( G^k )^{-1} & 1 \end{pmatrix} \\
       & = \begin{pmatrix} C_{k+1} & G^k \vect{w_{k+1}^\prime} \\ 
	    \transpose{\left( G^k \vect{s_{k+1}^\prime}\right)} & 1 \end{pmatrix}.
\end{align*}
Here, we used an obvious relation $\transpose{G} = G^{-1}$ in the last equality.
\par

Now, we set 
$$\vect{w_{k+1}} = G^k \vect{w_{k+1}^\prime} \, \text{ and }\, \vect{s_{k+1}} = G^k \vect{s_{k+1}^\prime}$$
so that we have
$$\Qt_{k+1}^{-1} \Xt_{k+1} \Qt_{k+1} 
  = \begin{pmatrix} C_{k+1} & \vect{w_{k+1}} \\ \transpose{\vect{s_{k+1}}} & 1 \end{pmatrix}$$
where $\vect{w_{k+1}}$, $\vect{s_{k+1}} \in \C^{2g}$ with either $\vect{w_{k+1}} = \vect{0}$ or 
$\vect{s_{k+1}} = \vect{0}$. It is easy to see that the entries of $\vect{w_{k+1}}$ and $\vect{s_{k+1}}$ 
are all zero except for the $(2k+1)$st through $(2k+4)$th rows.
\par

Finally, for $1 \leq i \leq k$, we note we may write 
$$C_i = \begin{pmatrix} \bar{C}_i & 0 \\ 0 & I_{2g-2k-2} \end{pmatrix} \, 
\text{ and }\, 
P_{k+1} = \begin{pmatrix} I_{2k+2} & 0 \\ 0 & \bar{P}_{k+1} \end{pmatrix}$$
for some $(2k+2) \times (2k+2)$ matrix $\bar{C}_i$ and $(2g-2k-2) \times (2g-2k-2)$ matrix
$\bar{P}_{k+1}$. We hence have
$$P_{k+1}^{-1} C_i P_{k+1} = C_i \, \text{ for $1 \leq i \leq k$.}$$
This implies for $1 \leq i \leq k$ that 
\begin{equation*}
 \Qt_{k+1}^{-1} \Xt_i \Qt_{k+1} = \Pt_{k+1}^{-1} ( \Qt_k^{-1} \Xt_i \Qt_k) \Pt_{k+1}
  = \begin{pmatrix} C_i & P_{k+1}^{-1} \vect{w_i} \\ \transpose{\vect{s_i}} P_{k+1} & 1 \end{pmatrix}.
\end{equation*}
For $i \leq k$, since the entries of $\vect{w_i}$ are all zero except for the first $2k+2$ rows, we can
easily see $P_{k+1}^{-1} \vect{w_i} = \vect{w_i}$. Similarly, we see for $i \leq k$, 
$$\transpose{\vect{s_i}} P_{k+1} = \transpose{(P_{k+1} \vect{s_i})} = \transpose{\vect{s_i}}.$$
\par

We can now conclude 
$$\Qt_{k+1}^{-1} \Xt_i \Qt_{k+1} 
  = \begin{pmatrix} C_i & \vect{w_i} \\ \transpose{\vect{s_i}} & 1 \end{pmatrix} \, 
  \text{ for $1 \leq i \leq k+1$}$$
where $\vect{w_i}$, $\vect{s_i} \in \C^{2g}$ with either $\vect{w_i} = \vect{0}$ or $\vect{s_i} = \vect{0}$,
and the entries of $\vect{w_i}$ and $\vect{s_i}$ are all zero except for the $(2i-1)$st through 
$(2i+2)$nd rows.
\ppar

Now, we apply this process repeatedly starting from $k=1$ to $k=g-2$. We then obtain 
nonzero complex numbers $p_1$, $p_2$, \ldots, $p_{g-1}$ so that for 
$P = Q_{g-1} = \diag{I_2, p_1 I_2, \ldots, p_{g-1} I_2}$ and 
$\Pt = \begin{pmatrix} P & 0 \\ 0 & 1 \end{pmatrix}$ have the desired property
$$\Pt^{-1} \Xt_k \Pt = \begin{pmatrix} C_k & \vect{w_k} \\ \transpose{\vect{s_k}} & 1 \end{pmatrix}\, 
  \text{ for $1 \leq k \leq g-1$}$$
where $\vect{w_k}$, $\vect{s_k} \in \C^{2g}$ with either $\vect{w_k} = \vect{0}$ or 
$\vect{s_k} = \vect{0}$. This completes the proof of Theorem \ref{prelim_thm:2g+1}. \qed
\ppar

\subsection{Proof of Theorem \ref{for_extra_gen}} \par

We first prove an analogue of Lemma \ref{key_lemma}.

\begin{lemma} \label{an_analogue}
 Let $g \geq 2$, $m \geq 1$, and $\Ft \in \GL{2g+m, \C}$. Let 
 $$\At_i = \begin{pmatrix} A_i & 0 \\ 0 & I_m \end{pmatrix} \quad\text{and}\quad
    \Bt_i = \begin{pmatrix} B_i & 0 \\ 0 & I_m \end{pmatrix} \quad\text{for $1 \leq i \leq g$.}$$
 Suppose $\Ft$  satisfies the following conditions (i)-(iv).
\begin{enumerate}
   \item[(i)] $\Ft$ has a unique eigenvalue $1$.
   \item[(ii)] $\Ft \At_i = \At_i \Ft$ \/ for \/ $1 \leq i \leq g$.
   \item[(iii)] $\Ft \Bt_j = \Bt_j \Ft$ \/ for \/ $2 \leq j \leq g$.
   \item[(iv)] $\Ft \Bt_1 \Ft = \Bt_1 \Ft \Bt_1$.
  \end{enumerate}
Then, it holds
 \begin{equation*}
  \Ft = \left(\begin{array}{cc} A_1 & W \\ S & T \end{array}\right)
 \end{equation*}
 where $T$ is an $m \times m$ matrix, and 
 $$
  W = \transpose{\begin{pmatrix} \vect{w} & \vect{0} & \vect{0} & \vect{0} & 
		 \vect{0} & \cdots & \vect{0} \end{pmatrix}}, \quad
  S = \begin{pmatrix} \vect{0} & \vect{s} & \vect{0} & \vect{0} & 
      \vect{0} & \cdots & \vect{0} \end{pmatrix}$$
 with $\vect{w}$, $\vect{s}$, $\vect{0} \in \C^m$; $\transpose{\vect{w}} \vect{s} = 0$; 
 $\transpose{\vect{w}}T = \transpose{\vect{w}}$, $T \vect{s} = \vect{s}$; 
and $T^2 - T = \vect{s} \transpose{\vect{w}}$.
\end{lemma}

\subsubsection*{Proof of Lemma \ref{an_analogue}}\par

The proof is basically the same as Lemma \ref{key_lemma}.
We write $\Ft = \begin{pmatrix} X & W \\ S & T \end{pmatrix}$ where $X$ and $T$ are, respectively,
$2g \times 2g$ and $m \times m$ matrices. Then the conditions (ii) to (iv), in turn, imply
\begin{itemize}
 \item \eqref{c_1}-\eqref{c_3} for $1 \leq i \leq g$;
 \item \eqref{c_4}-\eqref{c_6} for $2 \leq j \leq g$;
 \item \eqref{c_7}-\eqref{c_10} for $j=1$.
\end{itemize}

By the same argument for Lemma \ref{key_lemma}, by making use of \eqref{c_3}, \eqref{c_6} and
\eqref{c_2}, \eqref{c_5}, respectively, we obtain 
$$S = \begin{pmatrix} \vect{0} & \vect{s} & \vect{0} & \vect{0} & 
      \vect{0} & \cdots & \vect{0}\end{pmatrix}, \, 
  \text{ and } \, 
	 W = \transpose{\begin{pmatrix} \vect{w} & \vect{0} & \vect{0} & \vect{0} & 
			\vect{0} & \cdots & \vect{0} \end{pmatrix}}$$
for some $\vect{s}$, $\vect{w} \in \C^{m}$ where $\vect{0} \in \C^m$. Here, we note the entries in the 
fourth column of $S$ and the third row of $W$ must be zero because the condition (iii) is 
assumed for $j \geq 2$ rather than $j \geq 3$.
\par

Next, by the same argument for Lemma \ref{key_lemma} again, by using \eqref{c_1}, and \eqref{c_4} 
for $j \geq 3$ if $g \geq 3$, we obtain 
$$X = \begin{pmatrix} X_0 & 0 \\ 0 & I_{2g-4} \end{pmatrix}$$
where $X_0$ is the $4 \times 4$ matrix given by \eqref{form_X0} for some $\alpha$, $\beta$, $b$, $d \in \C$.
Furthermore, a direct computation shows that the equality \eqref{c_4} for $j=2$ implies 
$\alpha = \beta = d = 0$. Therefore, we may write
$$X_0 = \begin{pmatrix} I_2 + b (U - I_2) & 0 \\ 0 & I_{2g-4} \end{pmatrix}\, \text{ with $b \in \C$.}$$
Then, a straightforward computation shows that the equality \eqref{c_7} implies for 
$X_1 = I_2 + b (U - I_2)$, 
$$X_1 \Uh X_1 + \begin{pmatrix} 0 & \transpose{\vect{w}}\vect{s} \\ 0 & 0 \end{pmatrix} = \Uh X_1 \Uh.$$
By computing the matrix entries of the both sides, we can easily see 
$b=1$ and $\transpose{\vect{w}} \vect{s} = 0$. In particular, we have $X = A_1$.
\par

Finally, the equalities \eqref{c_8}-\eqref{c_10} for $j=1$ can be rewritten as 
\begin{itemize}
 \item $A_1 B_1 W + W T = B_1 W$,
 \item $S B_1 A_1 + T S = S B_1$,
 \item $S B_1 W + T^2 = T$.
\end{itemize}
These imply, in turn, 
$$\transpose{\vect{w}} T = \transpose{\vect{w}}, \quad T \vect{s} = \vect{s}, \, \text{ and } \,
  \vect{s} \transpose{\vect{w}} = T^2 - T.$$
This completes the proof of Lemma \ref{an_analogue}. \qed
\ppar

We can now prove Theorem \ref{for_extra_gen}. Suppose $\Ft \in \GL{2g+1, \C}$ satisfies 
the assumption of Theorem \ref{for_extra_gen}. Then we can apply Lemma \ref{an_analogue} with $m=1$
to obtain 
$$\Ft = \begin{pmatrix} A_1 & \vect{w} \\ \transpose{\vect{s}} & t \end{pmatrix}$$
where $\vect{w}$, $\vect{s} \in \C^{2g}$ and $t \in \C$ with the property given in the lemma, among which
\begin{align*}
 \vect{w} & = \transpose{\begin{pmatrix} w_1 & 0 & 0 & \cdots & 0 \end{pmatrix}},\\
 \vect{s} & = \transpose{\begin{pmatrix} 0 & s_2 & 0 & \cdots & 0 \end{pmatrix}}
\end{align*}
for some $w_1$, $s_2 \in \C$ with $w_1 s_2 = 0$. This implies either $\vect{w} = \vect{0}$ or 
$\vect{s} = \vect{0}$. We then have $t = 1$ from $\det{\Ft} = 1$. This completes the proof of 
Theorem \ref{for_extra_gen}. \qed
\par
 \section{A straightforward proof of Theorem \ref{2g-rep}} \label{straightforward} \par

The proof of Theorem \ref{2g-rep} given in \cite{sru} seems rather implicit in its final step, in the 
sense that it assumes without any reference to the literature 
that any injective homomorphism $f : \Sp{2g, \Z} \to \GL{2g, \C}$ is conjugate to the inclusion
	    $\Sp{2g, \Z} \hookrightarrow \GL{2g, \C}$ if it satisfies $f(A_i) = A_i$ and $f(B_i) = B_i$ 
	    for $1 \leq i \leq g$ where $\Sp{2g, \Z}$ is considered as the subgroup of $\GL{2g, \C}$ 
	    generated by $A_1$, $A_2$, \ldots, $A_g$, $B_1$, $B_2$, \ldots, $B_g$ given in 
	    Definition \ref{A_and_B} and $C_1$, $C_2$, \ldots, $C_{g-1}$ given in Definition \ref{L_and_Ck}.
One can avoid this and give a straightforward proof of Theorem \ref{2g-rep} by showing 
the following $2g$-dimensional analogue of Theorem \ref{prelim_thm:2g+1}, 
which is actually almost the same as proving the assumption just mentioned.
\par

\begin{prop}
 Let $g \geq 2$. For each $k$ with $1 \leq k \leq g-1$, suppose $X_k \in \GL{2g, \C}$ satisfies the 
 following conditions (i)-(iv).
 \begin{enumerate}
  \item[(i)] $X_k$ has exactly one eigenvalue $1$.
  \item[(ii)] $X_k A_i = A_i X_k$ for $i=1$, $2$, \ldots, $g$.
  \item[(iii)] If $g \geq 3$, then $X_k B_j = B_j X_k$ for each $j$ satisfying $1 \leq j \leq g$ and
	       $j \neq k$, $k+1$. 
  \item[(iv)] $X_k B_j X_k = B_j X_k B_j$ for $j=k$, $k+1$.
 \end{enumerate}
 Then, there exist nonzero complex numbers $p_2$, $p_3$, \ldots, $p_g$ such that for 
 $$P=\diag{I_2, p_2 I_2, p_3 I_2, \ldots, p_g I_2} \in \GL{2g, \C},$$ 
 it holds $P^{-1} A_i P = A_i$ and
 $P^{-1} B_i P = B_i$ for each $i=1$, $2$, \ldots, $g$, and furthermore, it also holds for 
 each $k=1$, $2$, \ldots, $g-1$, $P^{-1} X_k P = C_k$.
\end{prop}
This proposition does not follow directly from Theorem \ref{prelim_thm:2g+1}, but can be proved by 
the same line of arguments, which is simpler because of lower degree of matrix.
\ppar

After completing the work in this paper, the author was informed that Korkmaz revised and combined
his two papers \cite{ldl} and \cite{sru} into a single paper with the same title as \cite{ldl}, 
which contained the proof of the assumption mentioned above.
 \section{Appendix} \par \label{appendix}

In this appendix, we slightly generalize Lemma \ref{LemA}, which originated from Korkmaz \cite{sru} as
mentioned before. Rather complicated arguments below, together with Remark \ref{weakness},  
might suggest the limitation of the approach adopted in this paper. We first show:

\begin{lemma} \label{LemA-1}
 Let $g \geq 3$ and $m \geq 0$. Suppose
 $$\phi: \M{S} \to \GL{2g+m, \C}$$
 is an arbitrary homomorphism.
 Let $a$ be a non-separating simple closed curve on $S$, and $t_a$
 the right-handed Dehn twist along $a$.
 Let $\lambda$ be an eigenvalue of $L_a = \phi(t_a)$.
 \par
 If  $\lambda_\# \geq m+3$, then $\Dm{\ela} \geq 2g-2$. In particular,
 $\lambda_\# \geq 2g-2$. Furthermore, if $g \geq m+4$, then  $\Dm{\ela} \geq 2g-1$.
\end{lemma}
\par

 We remark that the case $m=0$ is nothing but Korkmaz \cite[Lemma 5.1]{sru}, and the case
 $m=1$ corresponds to Lemma \ref{LemA}.
 If $m \geq 1$,  the last consequence of the lemma can be strengthened to $\Dm{\ela} \geq 2g$, as
 we show in Proposition \ref{A-1ex}.
\ppar

\begin{proof}[Proof of Lemma \ref{LemA-1}]
 We first show $\Dm{\ela} \geq 2g-2$. To do so, assume to the contrary that $\Dm{\ela} \leq 2g-3$.
 \par

 Consider a regular neighborhood of $a$,  take the closure of its complement, and denote it by $S_a$.
 Note the genus of $S_a$ is $g-1 \geq 2$. The inclusion $S_a \hookrightarrow S$
 induces a homomorphism $\M{S_a} \to \M{S}$. Consider its  composition with $\phi$ and denote it by the
 same symbol as $\phi : \M{S_a} \to \GL{2g+m,\C}$.
 For each $i$ with $1 \leq i \leq \lambda_\#$, we set
 $W_i := \ker{\left( L_a - \lambda I \right)^i}$. We also set $W_0 = 0$ and $W_{\lambda_\# +1} = \C^{2g+m}$.
 Then, since the elements of $\M{S_a}$ commute with $t_a$, 
 we obtain an $\M{S_a}$-invariant flag
 $$0 = W_0 \subset W_1 \subset W_2 \subset \cdots \subset W_{\lambda_\#} \subset W_{\lambda_\# +1}=\C^{2g+m}.$$
 The dimension of each successive quotient $W_{i+1}/W_i$ ($0 \leq i \leq \lambda_\# - 1$)
 is equal to the number of the
 Jordan blocks for $L_a$ with eigenvalue $\lambda$ and with degree $\geq i+1$, which is at most the number
 of all the Jordan blocks with eigenvalue $\lambda$. We thus have $\Dm{W_{i+1}/W_i} \leq \Dm{\ela}$.
 Therefore, we have $\Dm{W_{i+1}/W_i} \leq 2g-3$. Furthermore, by the assumption $\lambda_\# \geq m+3$,
 we have $\Dm{W_{\lambda_\#+1}/W_{\lambda_\#}} = 2g+m-\lambda_\# \leq 2g-3$.
 Then by Theorem \ref{triviality} together with 
 Remark \ref{triviality-abelian}, $\phi(\M{S_a})$ is an abelian group. Hence $\phi$ is trivial 
 on the commutator subgroup of $\M{S_a}$.
 \par

 Next, choose a simple closed curve $a^\prime$ on $S_a$ so that $a^\prime$ is isotopic to $a$ in $S$.
 Since the genus of $S_a$ is at least two, we may choose simple closed curves $b$, $c$, $d$, $x$, $y$, $z$
 on $S_a$ which satisfy the lantern relation
 $$ t_{a^\prime} t_b t_c t_d = t_x t_y t_z.$$
 Then we may write  $t_{a^\prime} = (t_x t_b^{-1})(t_y t_c^{-1})(t_z t_d^{-1})$, and hence
 $t_{a^\prime}$ is contained in the commutator subgroup of $\M{S_a}$. Hence we have $\phi(t_{a^\prime}) = I$.
 Since $t_{a^\prime} = t_a$ in $\M{S}$, we have $L_a = I$, which contradicts to $\Dm{\ela} \leq 2g-3$.
 This shows $\Dm{\ela} \geq 2g-2$.
 \ppar

 Next, we prove the latter part of the lemma. Suppose $g \geq m+4$.
 Assume $\Dm{\ela} = 2g-2$ on the contrary.
 \par
 Let $b$ be a non-separating simple closed curve on $S$ which intersects with $a$
 transversely at a single point. Consider a regular neighbourhood of $a \cup b$ in the interior of $S$,
 and denote the closure of its complement in $S$ by $R$, whose genus is $g-1$. We divide into two cases
 according to whether $\ela$ coincides with $\elb$ or not.
 \par

 (I) The case $\ela = \elb$. Since $g \geq 2$, $\M{S}$ is generated by Dehn twists along non-separating
 simple closed curves. Therefore, by Theorem \ref{sru-lemma-4.3}, $\ela$ is $\M{S}$-invariant and has dimension
 $2g-2$. We also see $\Dm{\C^{2g+m}/\ela} = 2g+m - ( 2g-2 ) = m+2 \leq g-2 < 2g-1$. Therefore, we can apply
 Theorem \ref{triviality} to the $\M{S}$-invariant flag  $0 \subset \ela \subset \C^{2g+m}$ to see
 $\phi$ is trivial. This contradicts to $\Dm{\ela} = 2g-2$. 
 \par
 
 (II) The case $\ela \neq \elb$. Observe $\ela \cap \elb$ is $\M{R}$-invariant and its dimension satisfies
 $$ 2g - (m+4) \leq \Dm{\ela \cap \elb} \leq 2g-3.$$
 Therefore, we have a $\M{R}$-invariant flag
 $$0 \subset \ela \cap \elb \subset \C^{2g+m}$$
 which satisfies $\Dm{\ela \cap \elb} \leq 2g-3$ and $\Dm{\C^{2g+m}/(\ela \cap \elb)} \leq 2g-3$.
 Then by Theorem \ref{triviality}, $\phi(\M{R})$ is trivial, since $g-1 \geq 3$.
 Since $t_a$ is conjugate to an element of $\M{R}$ in $\M{S}$, we have $L_a = I$. This contradicts to
 $\Dm{\ela} = 2g-2$.
 \ppar
 This completes the proof.
 \end{proof}
\ppar

Finally, we show that the lower bound for $\Dm{\ela}$ can be improved by $1$, if $m \geq 1$.

 \begin{prop} \label{A-1ex}
  Suppose the assumption of Lemma \ref{LemA-1} with $g \geq m+4$.
  If $m \geq 1$, then it holds $\Dm{\ela} \geq 2g$.
 \end{prop}
 \par

\begin{proof}
 By Lemma \ref{LemA-1}, it is sufficient to confirm
 $\Dm{\ela} \neq 2g-1$. 
 Assume to the contrary that $\Dm{\ela} = 2g-1$.
 Let $b$ be a non-separating simple closed curve which intersects with $a$
 transversely at a single point. Choose a {\em separating} simple closed curve $c_0$ such that
 $c_0$ bounds a compact surface $R$ of genus $g-1$ with connected boundary and with
 no punctures so that the two curves $a$ and $b$ are contained in the complement of $R$.
 Since the genera of both $R$ and its complement are at least $1$, the inclusion $R \hookrightarrow S$
 induces an injective homomorphism $\M{R} \to \M{S}$ (\cite{paris-rolfsen}), via which we consider $\M{R}$
 as a subgroup of $\M{S}$. We divide into two cases according to whether $\ela = \elb$ or not.
 \ppar
 
 (I) The case $\ela = \elb$.
 Since $g \geq 2$, $\M{S}$ is generated by Dehn twists along non-separating simple closed curves.
 Therefore, $\ela$ is $\M{S}$-invariant by Theorem \ref{sru-lemma-4.3},
 and has dimension $2g-1$.
 We also have $\Dm{\C^{2g+m}/\ela} = m+1 < 2g-1$. Hence we can apply Theorem \ref{triviality} to the
 $\M{S}$-invariant flag
 $$0 \subset \ela \subset \C^{2g+m}$$
 to see that $\phi$ is trivial, since $g \geq 3$.
 This contradicts to $\Dm{\ela} = 2g-1$.
 \par

 (II) The case $\ela \neq \elb$.
 Since the elements of $\M{R}$ commute with both $t_a$ and $t_b$, $\ela \cap \elb$ is $\M{R}$-invariant,
 and its dimension satisfies $\Dm{\ela \cap \elb} \leq 2g-2$. 
 \par

 If $\Dm{\ela \cap \elb} < 2g-2$, then,
 since $\Dm{\ela \cap \elb} \leq 2g-3$ and
 $$\Dm{\C^{2g+m}/(\ela \cap \elb)} \leq 2m+2 \leq 2g-3,$$
 we can apply Theorem \ref{triviality} to the $\M{R}$-invariant flag $0 \subset \ela \cap \elb \subset \C^{2g+m}$
 to see that $\phi$ is trivial on $\M{R}$ since the genus of $R$ is at least $3$. Since $t_a$ is conjugate to an
 element of $\M{R}$, $L_a = I$, which contradicts to $\Dm{\ela} = 2g-1$.
 Therefore, we have $\Dm{\ela \cap \elb} = 2g-2$.
 \par

 Next, since $\Dm{\C^{2g+m}/(\ela \cap \elb)} = m+2 \leq 2g-3$, the action of $\M{R}$ on 
 $\C^{2g+m}/(\ela \cap \elb)$ induced by $\phi$ is trivial by Theorem \ref{fh-k}. On the other hand,
 since the genus of $R$ is
 at least $3$, Theorem \ref{2g-rep} implies that the action of $\M{R}$ via $\phi$ on $\ela \cap \elb$ is
 either trivial or conjugate to the symplectic representation $\rho_0^R: \M{R} \to \GL{H_1(\bar{R};\C)}$
 where $\bar{R}$ denotes the closed surface obtained from $R$ by gluing a $2$-disk along its boundary.
 If the action is trivial, we may take any basis of $\ela \cap \elb$ and extend it arbitrarily to a basis
 of $\C^{2g+m}$, according to which we have
 $$\phi(f) = \left( \begin{array}{c|c}
		       I_{2g-2} &  * \\
		       \hline
			0 & I_{m+2}
		      \end{array}
		      \right)$$
 for each $f \in \M{R}$. Therefore, $\phi(\M{R})$ is an abelian group. On the other hand, $\M{R}$ is perfect since
 the genus of $R$ is at least $3$. Hence $\phi(\M{R})$ is trivial. Again, since $t_a$ is conjugate in $\M{S}$
 to an element of $\M{R}$, we have $L_a = I$, which contradicts to $\Dm{\ela} = 2g-1$. Therefore, the only possible
 case is that the action of $\M{R}$ on $\ela \cap \elb$ is conjugate to the symplectic representation $\rho_0^R$.
 \par

 In this case, since $H_1(\bar{R};\C) = H_1(R;\C)$,
 we may choose an isomorphism $u: \ela \cap \elb \to H_1(R;\C)$ such that 
  $$u(\phi(f)\cdot v) = f_* u(v) \quad \text{($f \in \M{R}$, $v \in \ela \cap \elb$)}.$$
 Here, $f_*$ denotes the natural action of $f$ on $H_1(R;\C)$.
 \par
 We now fix a basis of $\C^{2g+m}$ extending an arbitrary basis of $\ela \cap \elb$. Then, under the
 identification of $\ela \cap \elb$ with $H_1(R;\C)$ via $u$, the image of $f \in \M{R}$ under $\phi$ has the form
    $$\phi(f) = \left( \begin{array}{c|c}
		       \rho_0^R(f) &  \begin{matrix} w_1 & w_2 & \cdots & w_{m+2} \end{matrix} \\
		       \hline
			0 & I_{m+2}
		      	\end{array}\right) \qquad \text{($w_1$, $w_2$, \ldots, $w_{m+2} \in H_1(R;\C)$).}$$
 For another $f^\prime \in \M{R}$ with
 $$\phi(f^\prime) = \left( \begin{array}{c|c}
		       \rho_0^R(f^\prime) &  \begin{matrix} w_1^\prime & w_2^\prime & \cdots & w_{m+2}^\prime \end{matrix} \\
		       \hline
			0 & I_{m+2}
		      	\end{array}\right) \qquad \text{($w_1^\prime$, $w_2^\prime$, \ldots, $w_{m+2}^\prime \in H_1(R;\C)$)},$$
we have
$$\phi(f f^\prime) = \left( \begin{array}{c|c}
		       \rho_0^R(f f^\prime) &  \begin{matrix} w_1 + f_* w_1^\prime & w_2 + f_* w_2^\prime & 
			                             \cdots & w_{m+2} + f_* w_{m+2}^\prime \end{matrix} \\
		       \hline
			0 & I_{m+2}
		      \end{array}\right) $$
 This formula shows for each $i$ that the correspondence $f \in \M{R} \mapsto w_i$ defines a crossed homomorphism
 $$c_i: \M{R} \to H_1(R;\C).$$
 \par

 Now let $d$ be a non-separating simple closed curve on $R$. We fix an orientation of $d$ and denote its representing
 homology class by $[\tilde{d}] \in H_1(R;\C)$. Then by Theorem \ref{crossed_hom}, there exists a complex
 number $z_i$ for each $i$ such that $c_i(t_d) = z_i [\tilde{d}]$. On the other hand, the action of $t_d$ is
 given by
 $$(t_d)_* x = x + \langle [\tilde{d}], x \rangle [\tilde{d}] \quad\text{for $x \in H_1(R;\C)$}$$
 where $\langle \cdot, \cdot \rangle$ denotes the algebraic intersection form on $H_1(R;\C)$. Therefore, 
we have
 $$\rank{(\phi(t_d)-I)} = 1.$$ 
Since $t_a$ is conjugate to $t_d$ in $\M{S}$, we can conclude $\rank{(L_a - I)} =1$,
 which contradicts to the assumption $\Dm{\ela} = 2g-1$.
 \par

 This completes the proof of Proposition \ref{A-1ex}.
\end{proof}
 
 \vspace{12pt}
\noindent
\textsc{Department of Mathematics,
  Kochi University of Technology \\ Tosayamada,  Kami City, Kochi 782-8502 Japan} \\
E-mail: {\tt kasahara.yasushi@kochi-tech.ac.jp}
\end{document}